\documentclass[a4paper,11pt]{article}

\usepackage[T1]{fontenc}
\usepackage{lmodern}

\usepackage{dsfont}

\usepackage[latin1]{inputenc}
\usepackage{amsmath}
\usepackage{amsthm}
\usepackage{amssymb}
\usepackage{mathrsfs}
\usepackage{graphicx}
\usepackage[all]{xy}
\usepackage{hyperref}

\usepackage{stmaryrd}
\usepackage{caption}

\usepackage{abstract} 

\newtheorem{thm}{Theorem}[section]
\newtheorem{cor}[thm]{Corollary}
\newtheorem{claim}[thm]{Claim}
\newtheorem{fact}[thm]{Fact}

\newtheorem{lemma}[thm]{Lemma}
\newtheorem{prop}[thm]{Proposition}

\theoremstyle{definition}
\newtheorem{definition}[thm]{Definition}
\newtheorem{ex}[thm]{Example}
\newtheorem{remark}[thm]{Remark}
\newtheorem{question}[thm]{Question}

\title{Hyperbolic and cubical rigidities of Thompson's group $V$}
\date{\today}
\author{Anthony Genevois}

\begin{document}

\maketitle

\begin{abstract}
In this article, we state and prove a general criterion allowing us to show that some groups are hyperbolically elementary, meaning that every isometric action of one of these groups on a Gromov-hyperbolic space either fixes a point at infinity, or stabilises a pair of points at infinity, or has bounded orbits. Also, we show how such a hyperbolic rigidity leads to fixed-point properties on finite-dimensional CAT(0) cube complexes. As an application, we prove that Thompson's group $V$ is hyperbolically elementary, and we deduce that it satisfies Property $(FW_{\infty})$, ie., every isometric action of $V$ on a finite-dimensional CAT(0) cube complex fixes a point. It provides the first example of a (finitely presented) group acting properly on an infinite-dimensional CAT(0) cube complex such that all its actions on finite-dimensional CAT(0) cube complexes have global fixed points.
\end{abstract}

\tableofcontents

\section{Introduction}

A major theme in geometric group theory is to make a given group act on a metric space which belongs to a specific class $\mathcal{C}$ in order to deduce some information about it. However, not every group is sensitive to a given class of spaces, meaning that every isometric action of a fixed group on any one of these spaces may turn out to be trivial in some sense. Nevertheless, although the machinery of group actions on spaces of $\mathcal{C}$ cannot be applied, it turns out that the non-existence of good actions provides interesting information as well. Roughly speaking, it implies some rigidity phenomena.

The first occurrence of such an idea was Serre's Property (FA). A group satisfies \emph{Property (FA)} if every isometric action on a simplicial tree fixes a point. We refer to \cite[\S 6]{MR1954121} for more information about this property. For instance, Property (FA) imposes restrictions on how to embed a given group into another (see for instance \cite{Fujiwara1999} in the context of 3-manifolds), and more generally on the possible homomorphisms between them (see for instance \cite[Corollary 4.37]{DrutuSapirActions} in the context of relatively hyperbolic groups). Also, such a rigidity has been applied in \cite{Herrlich} to determine when the fundamental groups of two graph of groups whose vertex-groups satisfy Property (FA) are isomorphic.

Another famous fixed-point property is Kazhdan's Property (T). Usually, Property (T) is defined using representation theory, but alternatively, one can say that a (discrete) group satisfies \emph{Property (T)} if every affine isometric action on a Hilbert space has a global fixed, or equivalently if every isometric action on a median space has bounded orbits. See \cite{BookT, medianviewpoint} for more information. Property (T) for a group imposes for instance strong restrictions on the possible homomorphisms starting from that group (for a geometric realisation of this idea, see for example \cite{PaulinOuter}, whose main construction
has been very inspiring in other contexts), and plays a fundamental role in several rigidity statements, including the famous Margulis' superrigidity. We refer to \cite{BookT}, and in particular to its introduction, for more information about Property (T).

In this article, we are mainly interested in the class of Gromov-hyperbolic spaces. We say that a group is \emph{hyperbolically elementary} if every isometric action on a hyperbolic space either fixes a point at infinity, or stabilises a pair of points at infinity, or has bounded orbits. Once again, such a property imposes restrictions on the possible homomorphisms between two groups. For instance, it is proved in \cite{HigherRankRigidity} that higher rank lattices are hyperbolically elementary, from which it is deduced that any morphism from a higher rank lattice to the mapping class group of a closed surface with punctures must have finite image (a statement originally due to Farb, Kaimanovich and Masur). It is worth noticing that a hyperbolically elementary group either satisfies Serre's Property (FA) or surjects onto $\mathbb{Z}$ or $D_{\infty}$, so that being hyperbolically is essentially a generalisation of Property (FA), which is not necessarily much harder to prove, see for instance \cite{DelzantGriBranch} about branch groups.

We emphasize the fact that it is not reasonable to remove the possibility of fixing a point at infinity from the definition of hyperbolic elementarity. Indeed, any infinite group admits a proper and parabolic action on a hyperbolic space; see for instance the classical construction explained in \cite[Section 4]{HruskaRH}. However, being hyperbolically elementary does not mean that any isometric action on a hyperbolic space is completely trivial, since the definition does not rule out lineal actions (ie., actions on a quasi-line) nor quasi-parabolic actions (ie., actions with loxodromic isometries all sharing a point at infinity). And these actions may provide interesting information on a group. For instance, admitting lineal actions is related to the existence of quasimorphisms; and admitting a quasi-parabolic action implies the existence of free sub-semigroups, so that the group must have exponential growth. 

The first main objective of our article is to prove a general criterion leading to some hyperbolic rigidity. More precisely:

\begin{thm}\label{main:criterion}
Let $G$ be a group. Suppose that there exist two subsets $A \subset B \subset G$ satisfying the following conditions.
\begin{itemize}
	\item $G$ is \emph{boundedly generated} by $A$, ie., there exists some $N \geq 0$ such that every element of $G$ is the product of at most $N$ elements of $A$.
	\item For every $a,b \in B$, there exist $g,h \in G$ such that $$[gag^{-1},a]=[gag^{-1},hgag^{-1}h^{-1}]=[hgag^{-1}h^{-1},b]=1.$$
	\item For every $a,b \in G$, there exist some $h, h_1, \ldots, h_r \in B$ such that the following holds. For every $k \in A$, there exists some $f \in \langle h_1 \rangle \cdots \langle h_r \rangle$ such that the elements $fkf^{-1}h$, $fkf^{-1}ha$ and $fkf^{-1}hb$ all belong to $B$.
\end{itemize}
Then any isometric action of $G$ on a hyperbolic space fixes a point at infinity, or stabilises a pair of points at infinity, or has bounded orbits. 
\end{thm}

%\noindent
%[Correction added on 9 december 2018 after online publication: The second item has been modified (producing a more general statement). The previous statement and its proof were correct, but the new formulation allowed us to correct a mistake in the proof of Theorem \ref{thm:VhypRigid} below.]

%\medskip
Our main motivating in proving this criterion is to show that Thompson's group $V$ is hyperbolically elementary.

\begin{thm}\label{main:Vhypelem}
Any isometric action of Thompson's group $V$ on a Gromov-hyperbolic space either fixes a point at infinity or has bounded orbits.
\end{thm}

The groups $F$, $T$ and $V$ were defined by Richard Thompson in 1965. Historically, Thompson's groups $T$ and $V$ are the first explicit examples of finitely presented simple groups. Thompson's groups were also used in \cite{ThompsonWP} to construct finitely presented groups with unsolvable word problems, and in \cite{ThompsonEmbeddings} to shows that a finitely generated group has a solvable word problem if and only if it can be embedded into a finitely generated simple subgroup of a finitely presented group. We refer to \cite{IntroThompson} for a general introduction to these three groups. Since then, plenty of articles have been dedicated to Thompson's groups, and they have been the source of inspiration for the introduction of many classes of groups, now referred to as \emph{Thompson-like groups}; see for instance \cite{HigmanSimple, BrownFiniteness, ThompsonStein, BV, FunarBraidedThompson, BrinHigher, BelkForrestFractals}. Nevertheless, Thompson's groups remain mysterious, and many questions are still open. For instance, it is a major open question to know whether $F$ is amenable, and the structure of subgroups of $V$ is still essentially unknown \cite{ObstructionsSubV}. 

Our initial motivation in proving Theorem \ref{main:Vhypelem} came from another fixed-point property, in the class of CAT(0) cube complexes. A group $G$
\begin{itemize}
	\item satisfies \emph{Property $(FW_n)$}, for some $n \geq 0$, if every isometric action of $G$ on an $n$-dimensional CAT(0) cube complex has a global fixed point;
	\item satisfies \emph{Property $(FW_{\infty})$} if it satisfies Property $(FW_n)$ for every $n \geq 0$;
	\item satisfies \emph{Property $(FW)$} if every isometric action of $G$ on a CAT(0) cube complex has a global fixed point. 
\end{itemize}
Property $(FW_{\infty})$ was introduced by Barnhill and Chatterji in \cite{BarnhillChatterji}\footnote{In \cite{BarnhillChatterji}, Barnhill and Chatterji named Property $(FW_{\infty})$ as Property $(FW)$. However, since then, Property $(FW)$ refers to the fixed-point property in CAT(0) cube complexes of arbitrary dimensions. So we changed the terminology.}, asking the difference between Kazhdan's Property (T) and Property $(FW_{\infty})$. It turns out that in general Property (T) is quite different from Property $(FW)$ or Property $(FW_{\infty})$. For instance, it is conjectured in \cite{CornulierLattices} that higher rank lattices satisfy Property $(FW)$, but such groups may be far from satisfying Property (T) since some of them are a-T-menable. For positive results in this direction, see \cite[Corollary 1.7]{CIF}, \cite[Example 6.A.7]{CornulierCommensurated}, \cite[Theorem 6.14]{CornulierLattices}. For some (very) recent developments related to Property $(FW)$, see \cite{CantatCornulier, LodhaFW, CornulierPiecewiseTransformations}.

The second main result of our article shows how to deduce Property $(FW_{\infty})$ from some hyperbolic rigidity. More explicitly: 

\begin{thm}\label{main:FW}
A finitely generated group all of whose finite-index subgroups 
\begin{itemize}
	\item are hyperbolically elementary,
	\item and do not surject onto $\mathbb{Z}$,
\end{itemize}
satisfies Property $(FW_{\infty})$.
\end{thm}

\noindent
We emphasize that the property of being hyperbolically elementary is not stable under taking finite-index subgroups, as shown by Example \ref{ex:Caprace}. Since Thompson's group $V$ is a simple group, the combination of our two main theorems immediately implies that $V$ satisfies Property $(FW_{\infty})$. 

\begin{cor}\label{cor:FWforV}
Any isometric action of Thompson's group $V$ on a finite-dimensional CAT(0) cube complex fixes a point.
\end{cor}

We emphasize that it was previously known that $V$ (as well as $F$ and $T$) does not act properly on a finite-dimensional CAT(0) cube complex. In fact, since $V$ contains a free abelian group of arbitrarily large rank, it follows that $V$ cannot act properly on any contractible finite-dimensional complex. 

Corollary \ref{cor:FWforV} contrasts with the known fact that $V$ acts properly on a locally finite infinite-dimensional CAT(0) cube complexes. (Indeed, Guba and Sapir showed in \cite[Example 16.6]{MR1396957} that $V$ coincides with the braided diagram group $D_b(\mathcal{P},x)$ where $\mathcal{P}$ is the semigroup presentation $\langle x \mid x^2=x \rangle$; and Farley constructed in \cite{FarleyPictureGroups} CAT(0) cube complexes on which such groups act.) As a consequence, $V$ provides another negative answer to \cite[Question 5.3]{BarnhillChatterji}, ie., $V$ is a new example of a group satisfying Property $(FW_{\infty})$ but not Property (T). Indeed, as a consequence of \cite{MR1459140}, a group acting properly on a CAT(0) cube complex does not satisfy Property (T); in fact such a group must be a-T-menable, according to \cite{Haagerup}, which is a strong negation of Kazhdan's Property (T).

So $V$ provides an example of a tough transition between finite and infinite dimensions, since on the one hand, $V$ has the best possible cubical geometry in infinite dimension: it acts properly on a locally finite CAT(0) cube complex; and on the other hand, it has the worst possible cubical geometry in finite dimension: every isometric action of $V$ on a finite-dimensional CAT(0) cube complex has a global fixed point. Using the vocabulary of \cite{CornulierCommensurated}, Thompson's group $V$ satisfies Property PW and Property $(FW_n)$ for every $n \geq 0$. It seems to be the first such example in the literature.

We would like to emphasize the fact that, although our article it dedicated to Thompson's group $V$, we expect that Theorem \ref{main:criterion} applies to most of the generalisations of $V$. For instance, without major modifications, our arguments apply to Higman-Thompson groups $V_{n,r}$ ($n \geq 2$, $r \geq 1$), to the group of interval exchange transformations $\mathrm{IET}([0,1])$, and to Neretin's group. However, since there does not exist a common formalism to deal with all the generalisations of $V$, we decided to illustrate our strategy by considering only $V$. Therefore, our paper should not be regarded as proving a specific statement about $V$, but as proposing a general method to prove hyperbolic and cubical rigidities of groups looking like $V$. In particular, we expect that our strategy works for higher dimensional Thompson's groups.

Finally, we would like to mention that Thompson's group $F$ is also hyperbolically elementary, since it does not contain any non-abelian free subgroup, but it does not satisfy Property $(FW_{\infty})$ since its abelianisation is infinite. About Thompson's group $T$, the situation is less clear, and our strategy does not work. So we leave it as an open question:

\begin{question}
Is Thompson's group $T$ hyperbolically elementary? Does it satisfy Property $(FW_{\infty})$?\footnote{Soon after this paper, Motoko Kato published an article on the ArXiv \cite{KatoFAFarb} proving that Thompson's group $T$ satisfies Property $(FW_\infty)$. More generally, she proved that Thompson's group $V$ and $T$ satisfy Farb's Property $F\mathcal{A}_k$ for every $k \geq 0$, meaning that every isometric action of $V$ or $T$ on a complete CAT(0) space of topological dimension $k$ has a global fixed point. Our question about the hyperbolic elementarity of $T$ remains open.}
\end{question}

The paper is organised as follows. First, Section \ref{section:preliminaries} is dedicated to basic definitions and preliminary lemmas about hyperbolic spaces and CAT(0) cube complexes. In Section \ref{section:reducible}, we introduce and study a family of particular elements of $V$, named \emph{reducible elements}. Finally, in Sections \ref{section:hyprigidity} and \ref{section:cubicalrigidity} respectively, we prove our general criteria, namely Theorems \ref{main:criterion} and \ref{main:FW}, and we prove Theorem \ref{main:Vhypelem} and Corollary \ref{cor:FWforV} by appling them to $V$. 

\paragraph{Acknowledgements.} I am grateful to Yves Cornulier, for his comments on an earlier version of this paper, which lead to a great improvement of the presentation; to Pierre-Emmanuel Caprace, for all his relevant comments; and to Sam Shepherd for having pointed out to me a mistake in the proof of Theorem \ref{thm:VhypRigid} in an earlier version. I also would like to thank the university of Vienna for its hospitality during the elaboration of this work. I was supported by the Ernst Mach Grant ICM-2017-06478, under the supervision of Goulnara Arzhantseva. Finally, I am grateful to the anonymous referee for his comments on the article.

\section{Preliminaries}\label{section:preliminaries}

\subsection{Hyperbolic spaces}\label{section:hyp}

\noindent
In this section, we recall some basic definitions about Gromov-hyperbolic spaces, we fix the notations which will be used in the paper, and we prove a few preliminary lemmas which will be useful later on. For more general information about hyperbolic spaces, we refer to \cite{GromovHyp, GhysdelaHarpe, CDP, BH}. 

\begin{definition}
Let $X$ be a metric space. For every $x,y,z \in X$, the \emph{Gromov product} $(x,y)_z$ is defined as
$$\frac{1}{2} \left( d(z,x)+d(z,y)-d(x,y) \right).$$
Fixing some $\delta \geq 0$, the space $X$ is \emph{$\delta$-hyperbolic} if the inequality
$$(x,z)_w \geq \min \left( (x,y)_w,(y,z)_w \right) - \delta$$
is satisfied for every $x,y,z,w \in X$. 
\end{definition}

\noindent
The following definitions will also be needed:
\begin{itemize}
	\item A map $f : X \to Y$ between two metric spaces is an \emph{$(A,B)$-quasi-isometric embedding}, where $A>0$ and $B \geq 0$, if $$\frac{1}{A} \cdot d(x,y)-B \leq d(f(x),f(y)) \leq A \cdot d(x,y) +B$$ for every $x,y \in X$. If $f$ is moreover every point of $Y$ is at distance at most $B$ from the image of $f$, then $f$ is an \emph{$(A,B)$-quasi-isometry}. A \emph{quasi-isometry} (resp. a \emph{quasi-isometric embedding}) is a map which an $(A,B)$-quasi-isometry (resp. an $(A,B)$-quasi-isometric embedding) for some $A >0$ and $B \geq 0$. 
	\item Given a metric space $X$ and two constants $A>0$ and $B \geq 0$, an \emph{$(A,B)$-quasigeodesic} is an $(A,B)$-quasi-isometric embedding from a segment of $\mathbb{R}$ or $\mathbb{Z}$ (depending on whether the metric of $X$ is discrete) to $X$. A \emph{quasigeodesic} is an $(A,B)$-quasigeodesic for some $A>0$ and $B \geq 0$.
	\item Given a geodesic metric space $X$ and a constant $K \geq 0$, a subspace $Y \subset X$ is \emph{$K$-quasiconvex} if every geodesic between two points of $Y$ stays in the $K$-neighborhood of $Y$. 
	\item Given a metric space $X$ and a subspace $Y \subset X$, the \emph{nearest-point projection} of a point $x \in X$ onto $Y$ is the set of all the points of $Y$ minimising the distance to $x$. The \emph{nearest-point projection} of another subspace $Z \subset X$ onto $Y$ is the union of all the nearest-point projections of the points of $Z$ onto $Y$. 
\end{itemize}

\noindent
Usually, it is easier to work with geodesic metric spaces instead of general metric spaces. The following lemma explains a classical trick which allows us to restrict our study to hyperbolic graphs.

\begin{lemma}\label{lem:hypgraph}
Let $X$ be metric space. If $Y$ denote the graph whose vertices are the points of $X$ and whose edges link two points at distance at most one, then the inclusion $X \subset Y$ is a $(1,0)$-quasi-isometry such that any isometry of $X$ extends uniquely to an isometry of $Y$. As a consequence, if $X$ is hyperbolic, then so is $Y$.
\end{lemma}

\noindent
From now on, all our (hyperbolic) metric spaces will be graphs.

\medskip \noindent
Fixing a graph $X$, three vertices $x,y,z \in X$ and a geodesic triangle $\Delta=[x,y] \cup [y,z] \cup [z,x]$, there exists a unique tripod $T$ and a unique map $f : \Delta \to T$ such that:
\begin{itemize}
	\item $f(x),f(y),f(z)$ are the endpoints of $T$;
	\item $f$ restricts to an isometry on each $[x,y]$, $[y,z]$, $[z,x]$.
\end{itemize}
The data $(T,f)$ is the \emph{comparison tripod} of $\Delta$, and the three (not necessarily distinct) points of $\Delta$ sending to the center of $T$ define the \emph{intriple} of $\Delta$.

\medskip \noindent
The following statement is an alternative definition of hyperbolic spaces (among geodesic metric spaces). We refer to the proof of \cite[Proposition 2.21]{GhysdelaHarpe} for more information. 

\begin{prop}
Let $X$ be a $\delta$-hyperbolic graph. For every vertices $x,y,z \in X$ and every geodesic triangle $\Delta=\Delta(x,y,z)$, if $(T,f)$ denotes the comparison tripod of $T$ then $d(a,b) \leq 4 \delta$ for every $a,b \in \Delta$ satisfying $f(a)=f(b)$. 
\end{prop}

\noindent
The next statement is a fundamental property satisfied by hyperbolic spaces, often referred to as \emph{Morse Property}. See for instance \cite[Theorem III.H.1.7]{BH}. 

\begin{thm}
Let $X$ be a $\delta$-hyperbolic graph. For every $A >0$ and every $B \geq 0$, there exists some $M(\delta,A,B)$, called the \emph{Morse constant}, such that: for every $x,y \in X$, any two $(A,B)$-quasigeodesics between $x$ and $y$ stay at Hausdorff distance at most $M(\delta,A,B)$. 
\end{thm}

\noindent
Now, let us prove two preliminary lemmas which will be useful in the next sections. 

\begin{lemma}\label{lem:geodinterproj}
Let $X$ be a $\delta$-hyperbolic graph and $\gamma_1,\gamma_2$ two lines which are $K$-quasiconvex for some $K \geq 0$. For every $x_1 \in \gamma_1$ and $x_2 \in \gamma_2$, any geodesic $[x_1,x_2]$ between $x_1$ and $x_2$ intersects the $M(\delta,1,2(4 \delta+K))$-neighborhood of the nearest-point projection of $x_1$ onto $\gamma_2$.  
\end{lemma}

\begin{proof}
Fix two points $x_1 \in \gamma_1$ and $x_2 \in \gamma_2$, and a geodesic $[x_1,x_2]$ between them. Let $p \in \gamma_2$ be a nearest-point projection of $x_1$ onto $\gamma_2$. Fixing some geodesics $[x_1,p]$ and $[p,x_2]$, we claim that $[x_1,p] \cup [p,x_2]$ is a $(1,2(4\delta+K)$-quasigeodesic.

\medskip \noindent
The only point to verify is that, given two points $a \in [x_1,p]$ and $b \in [p,x_2]$, the inequality
$$d(a,b) \geq d(a,p)+d(p,b) -2(4\delta+K)$$
holds. Let us consider a geodesic triangle $\Delta = \Delta(a,b,p)$, and let $\{q_1,q_2,q_3\}$ denote its intriple where $q_1 \in [a,b]$, $q_2 \in [a,p]$ and $q_3 \in [b,p]$. Notice that, since $\gamma_2$ is $K$-quasiconvex, there exists some $q \in \gamma_2$ satisfying $d(q_3,q) \leq K$. One has
$$\begin{array}{lcl} d(a,b) & = & d(a,q_1)+d(q_1,b) = d(a,q_2)+d(b,q_3) \\ \\ & = & d(a,p)-d(p,q_2) + d(p,b) - d(q_3,p) \\ \\ & = & d(a,p)+d(p,b)- 2d(p,q_2) \end{array}$$
On the other hand,
$$\begin{array}{lcl} d(x_1,q_2)+d(q_2,p) & = & d(x_1, \gamma_2) \leq d(x_1,q) \\ \\ & \leq & d(x_1,q_2) + d(q_2,q_3) + d(q_3,q) \\ \\ & \leq & d(x_1,q_2) + 4\delta +K \end{array}$$
hence $d(p,q_2) \leq 4\delta +K$. Our claim follows. We register our conclusion for future use.

\begin{fact}\label{fact:quasigeod}
Let $X$ be a $\delta$-hyperbolic graph, $\gamma$ a $K$-quasiconvex line, and $a \in X$, $b \in \gamma$ two vertices. If $p \in \gamma$ denotes a nearest-point projection of $a$ onto $\gamma$, then any concatenation $[a,p] \cup [p,b]$ defines a $(1,2(4\delta+K))$-quasigeodesic.
\end{fact}

\noindent
Now, we conclude from the Morse property that the Hausdorff distance between $[x_1,x_2]$ and $[x_1,p]\cup [p,x_2]$ is at most $M(\delta,1,2(4\delta+K))$. The desired conclusion follows. 
\end{proof}

\begin{lemma}\label{lem:brokengeod}
Let $X$ be a $\delta$-hyperbolic graph, $x,y \in X$ two vertices and $\gamma$ a $K$-quasiconvex line. Fix two nearest-point projections $x',y' \in \gamma$ respectively of $x,y$ onto $\gamma$, and suppose that $d(x',y') > 36\delta+5K$. Then
$$d(x,x')+d(x',y')+d(y,y')- 4(6\delta+K) \leq d(x,y) \leq d(x,x')+d(x',y')+d(y,y').$$
\end{lemma}

\begin{proof}
The right-hand side of our inequality is a consequence of the triangle inequality, so we only have to prove its left-hand side.

\medskip \noindent
Fix some geodesics $[x,y]$, $[x',y']$, $[x,x']$, $[y,y']$ and $[x',y]$. Let $\{p_1,p_2,p_3\}$ be the intriple of the geodesic triangle $\Delta(x,y,x')$ where $p_1 \in [x,x']$, $p_2 \in [x,y]$, $p_3 \in [x',y]$; and similarly let $\{q_1,q_2,q_3\}$ be the intriple of the geodesic triangle $\Delta(x',y',y)$ where $q_1 \in [y,y']$, $q_2 \in [x',y']$, $q_3 \in [x',y]$. Notice that, since $\gamma$ is $K$-quasiconvex, there exists some $q \in \gamma$ satisfying $d(q_2,q) \leq K$. The configuration is summarised by Figure \ref{figureCAT}. Notice that
$$\begin{array}{lcl} d(y,q_1)+d(q_1,y') & = & d(y,\gamma) \leq d(y,q) \\ \\ & \leq & d(y,q_1)+d(q_1,q_2) + d(q_2,q) \\ \\ & \leq & d(y,q_1)+ 4 \delta + K \end{array}$$
hence $d(y',q_1) \leq 4\delta +K$. Now, we distinguish two cases.
\begin{figure}
\begin{center}
\includegraphics[scale=0.16]{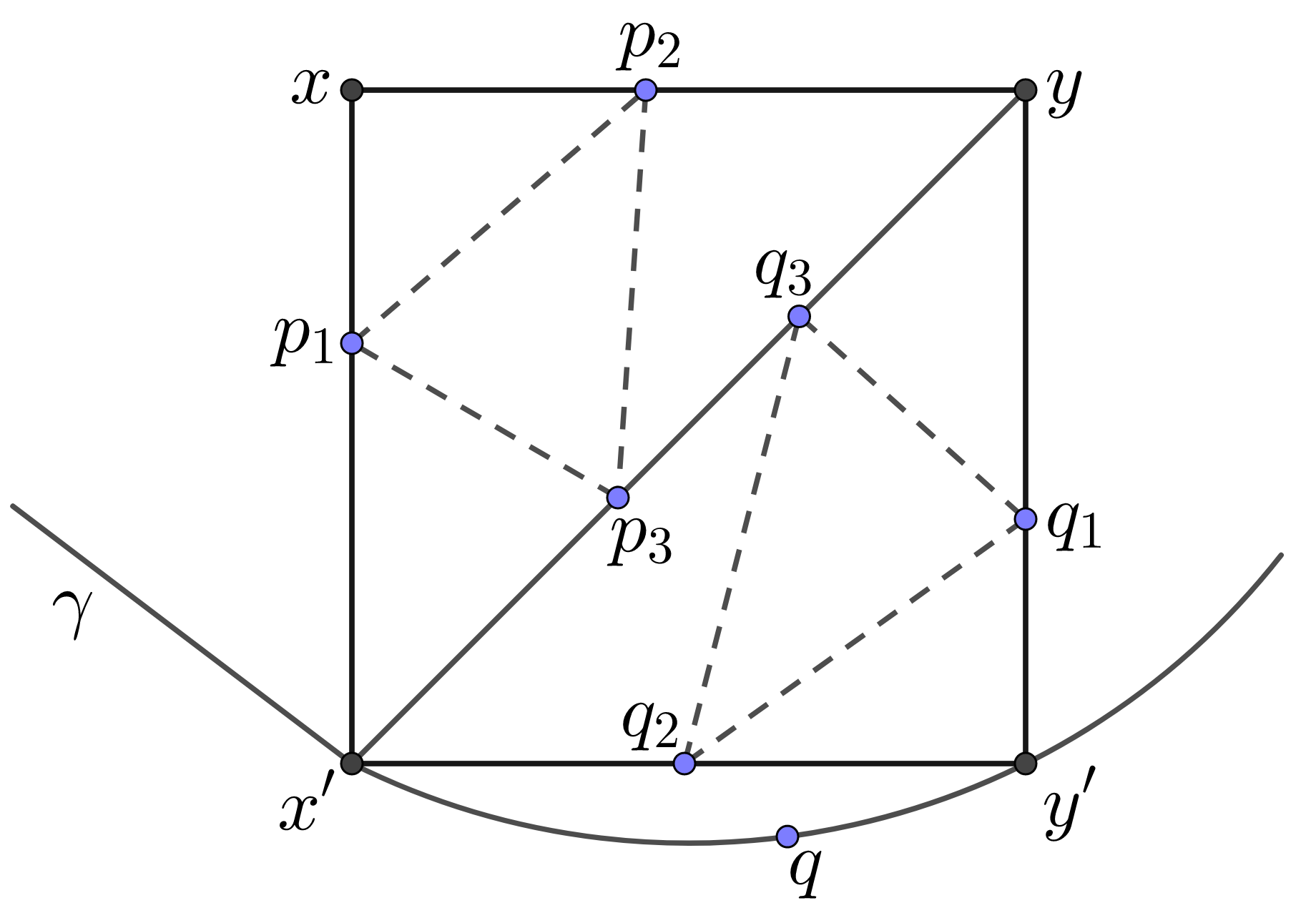}
\caption{The configuration of points in the proof of Lemma \ref{lem:brokengeod}.}
\label{figureCAT}
\end{center}
\end{figure}

\medskip \noindent
\textbf{Case 1:} Suppose that $d(x',p_3) \leq d(x',q_3)$. Then there exists some $p_3' \in [x',y']$ satisfying $d(p_3,p_3') \leq 4 \delta$, and because $\gamma$ is $K$-quasiconvex, there exists some $p_3'' \in \gamma$ satisfying $d(p_3',p_3'') \leq K$. One has
$$\begin{array}{lcl} d(x,p_1)+d(p_1,x') & =& d(x,x') = d(x,\gamma) \leq d(x,p_3'') \\ \\ & \leq & d(x,p_1)+d(p_1,p_3)+d(p_3,p_3')+d(p_3',p_3'') \\ \\ & \leq & d(x,p_1) + 4 \delta +4\delta +K \end{array}$$
hence $d(x',p_1) \leq 8\delta+K$. Next, notice that
$$\begin{array}{lcl} d(x,y) & = & d(x,p_2)+d(p_2,y) = d(x,p_1)+d(y,p_3) \\ \\ & = & d(x,x') - d(x',p_1) + d(y,x')-d(x',p_3) \\ \\ & =& d(x,x') + d(y,x')-2d(x',p_1) \end{array}$$
and that
$$\begin{array}{lcl} d(y,x') & = & d(y,q_3) + d(q_3,x') = d(y,q_1) + d(x',q_2) \\ \\ & = & d(y,y')-d(y',q_1) + d(x',y') - d(y',q_2) \\ \\ & = & d(y,y')+d(x',y') - 2d(y',q_1) \end{array}$$
We conclude that
$$\begin{array}{lcl} d(x,y) & = & d(x,x')+d(x',y')+d(y,y') - 2 \left( d(x',p_1)+d(y',q_1) \right) \\ \\ & \geq & d(x,x')+d(x',y')+d(y,y') -4(6\delta+K) \end{array}$$
\textbf{Case 2:} Suppose that $d(x',p_3)> d(x',q_3)$. As a consequence, there exists some $q_3' \in [x,x']$ satisfying $d(q_3,q_3') \leq 4\delta$. Notice that 
$$\begin{array}{lcl} d(x,q_3')+d(q_3',x') & = & d(x, \gamma)  \leq d(x,q) \leq d(x,q_3') + d(q_3',q_3)+d(q_3,q_2)+d(q_2,q) \\ \\ & \leq & d(x,q_3') + 4\delta + 4 \delta +K\end{array}$$
hence $d(q_3',x') \leq 4(6 \delta +K)$. Therefore, 
$$\begin{array}{lcl} d(x',y') & \leq & d(x',q_3')+d(q_3',q_3)+d(q_3,q_1)+d(q_1,y) \\ \\ & \leq & 4(6\delta+K)+4\delta + 4 \delta + (4\delta+K) \end{array}$$
ie., $d(x',y') \leq 36\delta+5K$. This contradicts our assumptions, so our second case cannot happen. 
\end{proof}

\begin{cor}\label{cor:foot}
Let $X$ be a $\delta$-hyperbolic graph, $x,y \in X$ two vertices, $[x,y]$ a geodesic between $x$ and $y$, and $\gamma$ a $K$-quasiconvex line. Fix two nearest-point projections $x',y' \in \gamma$ respectively of $x,y$ onto $\gamma$, and suppose that $d(x',y') > 36\delta+5K$. Then $d(x,y) \geq d(x,x')$, so that there exists a unique point $a \in [x,y]$ satisfying $d(x,a)=d(x,x')$, and moreover $d(a,x') \leq 2M(\delta, 1,4(6\delta+K))$.
\end{cor}

\begin{proof}
First of all, notice that, as a consequence of Lemma \ref{lem:brokengeod}, one has
$$\begin{array}{lcl} d(x,y) & \geq & d(x,x')+d(x',y')+d(y',y) - 4(6 \delta +K) \\ \\ & > & d(x,x') + 36\delta+5K - 4(6 \delta +K) \geq d(x,x') \end{array}$$
which proves the first assertion of our statement.

\medskip \noindent
Fix some geodesics $[x,x']$, $[y,y']$, $[x',y']$. As a consequence of Fact \ref{fact:quasigeod} and Lemma \ref{lem:brokengeod}, we know that $[x,x'] \cup [x',y'] \cup [y',y]$ defines a $(1,4(6\delta+K))$-quasigeodesic between $x$ and $y$. It follows from the Morse property that there exists some $a' \in [x,y]$ satisfying $d(x',a') \leq M(\delta, 1, 4(6\delta+K))$. One has
$$d(a,x') \leq d(a,a')+d(a',x') \leq d(a,a')+ M(\delta,1,4(6\delta+K)).$$
On the other hand,
$$|d(x,a')-d(x,x')| \leq d(a',x') \leq M(\delta,1,4(6 \delta+K)),$$
so
$$d(a,a')= |d(x,a)-d(x,a')| = |d(x,x')-d(x,a')| \leq M(\delta,1,4(6\delta+K)).$$
The desired conclusion follows. 
\end{proof}

\begin{cor}\label{cor:distproj}
Let $X$ be a $\delta$-hyperbolic graph, $\gamma$ a $K$-quasiconvex line and $x,y \in X$ two points. If $x',y' \in \gamma$ are nearest-point projections onto $\gamma$ of $x,y$ respectively, then
$$d \left( x',y' \right) \leq d(x,y) + 36 \delta+5K.$$
\end{cor}

\begin{proof}
If $d(x',y') \leq 36\delta +5K$ there is nothing to prove, so suppose that $d(x',y') >36\delta+5K$. As a consequence of Lemma \ref{lem:brokengeod},
$$d(x,y) \geq d(x',y') - 4(6\delta+K) \geq d(x',y') - 36\delta-5K,$$
which concludes the proof of our corollary. 
\end{proof}

\noindent
Now, let $X$ be a $\delta$-hyperbolic graph and $g \in \mathrm{Isom}(X)$ an isometry. The \emph{translation length} of $g$ is 
$$[g]= \inf \left\{ d(x,g \cdot x) \mid x \in X \right\};$$
and the \emph{minimal set} of $g$ is
$$C_g = \left\{ x \in X \mid d(x,g \cdot x) = [g] \right\}.$$ 
It is worth noticing that, because $X$ is a graph, the infinimum in the definition of $[g]$ turns out to be a minimum, so that $C_g$ is non-empty. 

\begin{definition}
Let $X$ be a hyperbolic graph and $g \in \mathrm{Isom}(X)$ a loxodromic isometry. An \emph{axis} of $g$ is a concatenation $\ell= \bigcup\limits_{k \in \mathbb{Z}} g^k \cdot [x,g \cdot x]$ for some $x \in C_g$. 
\end{definition}

\noindent
Noticing that an axis of $g$ is a $[g]$-local geodesic, the following lemma follows from \cite[Theorem III.$\Gamma$.1.13]{BH}.

\begin{lemma}
Let $X$ be a $\delta$-hyperbolic graph and $g \in \mathrm{Isom}(X)$ a loxodromic isometry satisfying $[g] >32 \delta$. Any axis of $g$ is $12\delta$-quasiconvex.
\end{lemma}

\noindent
We conclude this section with a last preliminary lemma, which will be fundamental in the proof of the hyperbolic rigidity of Thompson's group $V$.

\begin{lemma}\label{lem:hypgeom}
Let $X$ be a hyperbolic graph and $g,h \in \mathrm{Isom}(X)$ two isometries. Suppose that $g$ is loxodromic of translation length at least $525\delta$ and that $h$ is elliptic. Fix an axis $\ell$ of $g$. If $hg$ is elliptic, then there exists a point $x \in \ell$ such that 
$$d(x,hx) \leq 8M(\delta, 1, 62\delta)+243\delta.$$
\end{lemma}

\begin{proof}
For convenience, fix a $G$-equivariant map $\pi : X \to \ell$ sending every point of $X$ to one of its nearest-point projections, and set $M= 2M(\delta,1,62\delta)$. Because $h$ is elliptic, we know from \cite[Lemma III.$\Gamma$.3.3]{BH} that there exists some $x \in X$ such that $\langle h \rangle \cdot x$ has diameter at most $17 \delta$. 

\medskip \noindent
First, suppose that there exists some $y \in  X$ such that the distances $d(\pi(x),\pi(y))$ and $d(\pi(hy),\pi(hx))$ are both greater than $96\delta$. Fix a geodesic $[x,y]$. We know from Corollary \ref{cor:foot} that there exists a point $z \in [x,y]$ satisfying $d(x,z)=d(x,\pi(x))$ such that $d(z,\pi(x)) \leq 2M$. Similarly, we know from Corollary \ref{cor:foot} that there exists a point $w \in h \cdot [x,y]$ satisfying $d(hx,w)=d(hx,\pi(hx))$ such that $d(w,\pi(hx)) \leq 2M$. Notice that
$$\begin{array}{lcl} d(hz,w) & = & |d(hx,w)- d(hx,hz)| = | d(hx,\pi(hx)) - d(x, \pi(x))| \\ \\ & \leq & d(x,hx)+d(\pi(x),\pi(hx)) \leq 130\delta \end{array}$$
where the last inequality is justified by Corollary \ref{cor:distproj}. Consequently,
$$\begin{array}{lcl} d(z,hz) & \leq & d(z,\pi(x)) + d(\pi(x),\pi(hx)) + d(\pi(hx),w)+d(w,hz) \\ \\ & \leq & 2M + (17\delta + 96 \delta) +2M + 130 \delta = 4M+243 \delta \end{array}$$
We conclude that $\pi(x)$ is a point satisfying the conclusion of our lemma, since
$$d(\pi(x),h \pi(x)) \leq d(z,hz)+2d(z,\pi(x)) \leq 8M+243\delta.$$
Next, suppose that for every $y \in X$ satisfying $d(\pi(x),\pi(y))>96\delta$ one has $d(\pi(hy),\pi(hx)) \leq 96\delta$. Consequently, since 
$$d(\pi(x),\pi(gx))=d(\pi(x),g \pi(x)) = [g] >96\delta,$$ 
it follows that $d(\pi(hx),\pi(hgx)) \leq 96\delta$. So
$$\begin{array}{lcl} d(\pi(gx),\pi(ghgx)) & = & d(\pi(x), \pi(hgx)) \leq d(\pi(x),\pi(hx)) + d(\pi(hx),\pi(hgx)) \\ \\ & \leq & d(x,hx) + 96\delta + 96\delta \leq 209\delta \end{array}$$
where the first inequality of the second line is justified by Corollary \ref{cor:distproj}. Next, since
$$d(\pi(x),\pi(ghgx)) \geq d(\pi(x),\pi(gx))- d(\pi(gx),\pi(ghgx)) \geq [g] - 209\delta > 96\delta,$$
we deduce from Lemma \ref{lem:brokengeod} that
$$\begin{array}{lcl} d(x,ghgx) & \geq & d(x,\pi(x))+d(\pi(x),\pi(ghgx)) +d(\pi(ghgx),ghgx) - 72 \delta. \end{array}$$
By noticing that
$$d(\pi(x), \pi(ghgx)) \geq d(\pi(x),\pi(gx))-d(\pi(gx),\pi(ghgx) \geq [g] - 209\delta$$
and that
$$\begin{array}{lcl} d(\pi(ghgx),ghgx) & = & d(\pi(hgx),hgx) \geq d(hgx,x)-d(x, \pi(hgx)) \\ \\ & \geq & d(x,hgx) - d(x,\pi(x))-d(\pi(x),\pi(hgx)) \\ \\ & \geq & d(x,hgx)-d(x, \pi(x)) - 209\delta \end{array}$$
the previous inequality becomes
$$\begin{array}{lcl} d(x,ghgx) & \geq & d(x,\pi(x)) + [g]-209\delta +d(x,hgx)-d(x,\pi(x)) - 209\delta - 72 \delta \\ \\ & \geq & d(x,hgx) + [g] -490 \delta \end{array}$$
hence
$$d(x,(hg)^2x) \geq d(x,ghgx)-d(x,hx) \geq d(x,hgx) + [g] - 507\delta.$$
Since $[g]>525\delta$ by assumption, it follows that
$$d(x,(hg)^2x) > d(x,hgx)+18 \delta.$$
According to \cite[Corollaire 8.22]{GhysdelaHarpe}, this inequality implies that $hg$ is loxodromic, contradiction our hypotheses.
\end{proof}

%\begin{lemma}\label{lem:diamproj}
%Let $X$ be a $\delta$-hyperbolic graph, $C \subset X$ a $K$-quasiconvex subspace, and $x \in X$ a point. If $a,b \in C$ are two nearest-point projections of $x$ onto $C$, then $d(a,b) \leq 2(8 \delta+K)$. 
%\end{lemma}
%
%\begin{proof}
%Fix three geodesics $[x,a]$, $[x,b]$, $[a,b]$; and let $m$ denote the midpoint of $[a,b]$. Because $C$ is $K$-quasiconvex, there exists some $c \in C$ such that $d(m,c) \leq K$. By $\delta$-hyperbolicity, we also know that there exists some $m' \in [x,a]\cup [x,b]$ satisfying $d(m,m') \leq 4 \delta$. Without loss of generality, suppose that $m' \in [x,a]$. One has
%$$\begin{array}{lcl} d(x,m')+d(m',a) & = & d(x,a)=d(x,C) \leq d(x,c) \leq d(x,m')+d(m',m) + d(m,c) \\ \\ & \leq & d(x,m') + 4 \delta+K \end{array}$$
%hence $d(a,m') \leq 4\delta+K$. It follows that
%$$d(a,m) \leq d(m,m')+d(m',a) \leq 8 \delta+K,$$
%so $d(a,b) \leq 2(8 \delta+K)$, the desired conclusion.
%\end{proof}

\subsection{CAT(0) cube complexes}\label{section:cubecomplexes}

\noindent
In this paper, we suppose that the reader is familiar with the basic definitions and properties of CAT(0) cube complexes. For details, we refer to \cite{SageevCAT(0), WiseBook}. Nevertheless, we recall the following fundamental property of cubical complexes, which will be used several times in Section \ref{section:cubicalrigidity} without mentioning it. We refer to \cite[Theorem 11.9]{Roller} for a proof. 

\begin{thm}
Let $G$ be a group acting on some CAT(0) cube complex $X$. If $G$ has a bounded orbit, then $G$ stabilises a cube. As a consequence, the action has a global fixed point.
\end{thm}

\noindent
The rest of this section is dedicated to some properties of Roller boundary of CAT(0) cube complexes and of the hyperbolic model introduced in \cite{MoiHypCube}. These statements will be useful in Section \ref{section:cubicalrigidity}.

\paragraph{Roller boundary.} Let $X$ be a CAT(0) cube complex. An \emph{ultrafilter} $\sigma$ is a collection of halfspaces of $X$ such that
\begin{itemize}
	\item $\sigma$ contains exactly one of the two halfspaces delimited by a given hyperplane;
	\item if $D_1$ and $D_2$ are two halfspaces satisfying $D_1 \subset D_2$, then $D_1 \in \sigma$ implies $D_2 \in \sigma$.
\end{itemize}
For every vertex $x \in X$, the collection $\sigma_x$ of all the halfspaces of $X$ containing $x$ is the \emph{principal ultrafilter} defined by $x$. 

\medskip \noindent
The \emph{Roller compactification} of $X$ is the graph $\overline{X}$ whose vertices are the ultrafilters of $X$ and whose edges link two ultrafilters whenever their symmetric difference has cardinality two. The Roller compactification is usually not connected, but each connected component turns out to be a median graph (which we identify canonically with a CAT(0) cube complex; see \cite{mediangraphs}). Moreover, the map $x \mapsto \sigma_x$ defines an embedding $X^{(1)} \hookrightarrow \overline{X}$ whose image is a connected component of $\overline{X}$. We refer to the connected components of $\overline{X}$ as its \emph{cubical components}, and we identify $X$ with the cubical component of the principal ultrafilters. The \emph{Roller boundary} of $X$ is $\mathfrak{R}X := \overline{X} \backslash X$. 

\medskip \noindent
Finally, we define a topology on $\overline{X}$, and a fortiori on $\mathfrak{R}X$, as follows. By labelling the two halfspaces delimited by a given hyperplane with $0$ and $1$, we can naturally thought of $\overline{X}$ as a subset of $\{0,1\}^{\mathcal{H}}$, where $\mathcal{H}$ denotes the set of all the hyperplanes of $X$. The topology of $\overline{X}$ is the topology induced by the product topology on $\{0,1\}^{\mathcal{H}}$. Since $\overline{X}$ is closed in $\{0,1\}^{\mathcal{H}}$, it follows that $\overline{X}$ is compact. More details about Roller boundary can be found in \cite{SageevCAT(0), Roller}.

\medskip \noindent
The following statement provides a useful trick when arguing by induction on the dimension.

\begin{lemma}\label{lem:dimcubicalcomp}
Let $X$ be a finite-dimensional CAT(0) cube complex. For every cubical component $Y \subset \mathfrak{R}X$, the inequality $\dim(Y)<\dim(X)$ holds.
\end{lemma}

\noindent
A proof can be found for instance in \cite[Proposition 4.29]{RollerMedian}, in the more general context of median spaces. 

\paragraph{Hyperbolic model of cube complexes.} In \cite{MoiHypCube}, we introduced a hyperbolic model (depending on a parameter) of CAT(0) cube complexes. Below, we recall the first definitions and properties, and we prove a proposition related to its Gromov-boundary. 

\begin{definition}
Let $X$ be a CAT(0) cube complex and $L \geq 0$ an integer. A \emph{facing triple} is the data of three pairwise disjoint hyperplanes such that no one separates the other two. Two hyperplanes $J_1,J_2$ are \emph{$L$-well-separated} if they are not transverse and if every collection of hyperplanes transverse to both $J_1$ and $J_2$ which does not contain any facing triple has cardinality at most $L$. An isometry $g \in \mathrm{Isom}(X)$ is \emph{$L$-contracting} if it \emph{skewers} a pair of $L$-well-separated hyperplanes, ie., if there exist two $L$-well-separated hyperplanes $J_1,J_2$ delimiting two halfspaces $D_1,D_2$ respectively such that $g D_2 \subsetneq D_1 \subset D_2$.
\end{definition}

\noindent
The terminology ``$L$-contracting isometry'' is justified as follows. In an arbitrary metric space $X$, a \emph{contracting isometry} usually refers to an isometry $g \in \mathrm{Isom}(X)$ such that there exists some point $x_0 \in X$ satisfying the following two conditions:
\begin{itemize}
	\item the orbit map $n \mapsto g^n \cdot x_0$ defines a quasi-isometric embedding $\mathbb{Z} \hookrightarrow X$;
	\item there exists some $B \geq 0$ such that the nearest-point projection of any ball disjoint from $\langle g \rangle \cdot x_0$ onto $\langle g \rangle \cdot x_0$ has diameter at most $B$.
\end{itemize}
In \cite[Theorem 3.13]{article3}, we proved the following characterisation:

\begin{prop}
Let $X$ be a CAT(0) cube complex. An isometry $g \in \mathrm{Isom}(X)$ is contracting if and only if there exists some $L \geq 0$ such that $g$ is $L$-contracting.
\end{prop}

\noindent
Consequently, our terminology agrees with the usual terminology which can be found in the litterature.

\medskip \noindent
Given a CAT(0) cube complex $X$ and an integer $L \geq 0$, one next defines a new metric on (the set of vertices of) $X$ by:
$$\delta_L : (x,y) \mapsto \begin{array}{c} \text{maximal number of pairwise $L$-well-separated} \\ \text{hyperplanes separating $x$ and $y$} \end{array}$$
We showed in \cite{MoiHypCube} that $\delta_L$ is indeed a metric, and we proved the following statement:

\begin{thm}\label{thm:hypmodel}
Let $X$ be a CAT(0) cube complex and $L \geq 0$ some integer. The metric space $(X,\delta_L)$ is hyperbolic, and an isometry of $X$ defines a loxodromic isometry of $(X,\delta_L)$ if and only if it $L$-contracting. 
\end{thm}

\noindent
In the rest of the section, we would like link the Gromov-boundary of $(X,\delta_L)$ with the Roller boundary of $X$. Notice that it is not clear whether $(X,\delta_L)$ is a geodesic metric space, so, given a basepoint $x_0 \in X$, the boundary will be defined as the quotient of the collection of sequences $(x_i)$ satisfying $(x_i,x_j)_{x_0} \underset{i,j \to + \infty}{\longrightarrow} + \infty$ modulo the equivalence relation: $(y_i) \sim (z_i)$ if $(y_i,z_i)_{x_0} \underset{i,j \to + \infty}{\longrightarrow} + \infty$. (Nevertheless, it follows from \cite[Lemma 6.55]{MoiHypCube} that $(X,\delta_L)$ is a quasigeodesic metric space, so the boundary can also be defined as the asymptotic classes of quasigeodesic rays.) Our main statement is:

\begin{prop}\label{prop:boundarymap}
Let $X$ be a CAT(0) cube complex and $L \geq 0$ an integer. There exists an $\mathrm{Isom}(X)$-equivariant map sending a point of $\partial(X,\delta_L)$ to a subset of diameter at most $L$ in a cubical component of $\mathfrak{R}X$ .
\end{prop}

\noindent
First, we recall \cite[Lemma 6.55]{MoiHypCube}, which essentially states that the quasigeodesics in $(X,\delta_L)$ fellow-travel the geodesics in $X$. 

\begin{lemma}\label{lem:deltaI}
Let $X$ be a CAT(0) cube complex and $x,y,z \in X$ three vertices such that $z$ belongs to a geodesic between $x$ and $y$ in $X$. Then
$$\delta_L(x,z)+\delta_L(z,y)- 2(L+3) \leq \delta_L(x,y) \leq \delta_L(x,z)+\delta_L(z,y).$$
\end{lemma}

\noindent
As a consequence of the previous lemma, we are able to estimate the Gromov product in $(X,\delta_L)$. (In the following, Gromov products will always refer to the distance $\delta_L$.)

\begin{lemma}\label{lem:GromovProductEst}
Let $X$ be a CAT(0) cube complex, $L \geq 0$ an integer and $x,y,z \in X$ three vertices. Then
$$\left| (x,y)_z - \delta_L(z,m(x,y,z)) \right| \leq 3(L+3),$$
where $m(x,y,z)$ denotes the median point of $x,y,z$. 
\end{lemma}

\begin{proof}
For convenience, set $m=m(x,y,z)$. By applying Lemma \ref{lem:deltaI}, we get
$$\begin{array}{lcl} \left| (x,y)_z - \delta_L(z,m(x,y,z)) \right| & = & \frac{1}{2} \left| \delta_L(z,x)+\delta_L(z,y)- \delta_L(x,y) - 2\delta_L(z,m) \right| \\ \\ & \leq & \frac{1}{2} \left| \delta_L(z,m)+\delta_L(m,x)+ \delta_L(z,m)+ \delta_L(m,y) \right. \\ & & \left. - \delta_L(x,m)- \delta_L(m,y) - 2 \delta_L(z,m) \right| + 3(L+3) \\ \\ & \leq & 3(L+3)\end{array}$$
which concludes the proof.
\end{proof}

\begin{proof}[Proof of Proposition \ref{prop:boundarymap}.]
For every $\xi \in \partial (X,\delta_L)$, we denote by $R(\xi)$ the set of all the accumulation points in $\overline{X}$ of all the sequence of vertices representing $\xi$. We want to prove that $\xi \mapsto R(\xi)$ is the map we are looking for. First of all, notice that $R(\xi)$ is non-empty for every $\xi \in \partial(X,\delta_L)$, as a consequence of the compactness of $\overline{X}$, and that our map is clearly $\mathrm{Isom}(X)$-equivariant. 

\medskip \noindent
Next, we claim that $R(\xi) \subset \mathfrak{R}X$ for every $\xi \in \partial (X,\delta_L)$. Indeed, let $(x_i)$ be a sequence representing $\xi$ and $z \in \overline{X}$ one of its accumulation points. For convenience, suppose that $(x_i)$ converges to $z$ in $\overline{X}$. Because
$$d(x_0,x_i) \geq \delta_L(x_0,x_i) = \delta_L(x_0,m(x_0,x_i,x_i)) \geq (x_i,x_i)_{x_0}- 3(L+3) \underset{i \to + \infty}{\longrightarrow} + \infty,$$
where the last inequality is justified by Lemma \ref{lem:GromovProductEst}, it is clear that $z$ cannot belong to $X$, so it must belong to $\mathfrak{R}X$.

\medskip \noindent
Finally, we need to verify that, given some $\xi \in \partial (X,\delta_L)$, if $(y_i)$ and $(z_i)$ are two sequences representing $\xi$ and converging respectively to $y$ and $z$ in $\overline{X}$, then $y$ and $z$ belong to the same cubical component and there the distance between them is at most $L$.

\medskip \noindent
Let $J_1, \ldots, J_k$ be $k$ hyperplanes such that, for every $1 \leq i \leq k$, the ultrafilters $y$ and $z$ does not contain the same halfspace delimited by $J_i$. Set $D= \max\limits_{1 \leq i \leq k} d(x_0,J_i)$. By definition of the topology of $\overline{X}$, there exists some $N \geq 1$ such that, for every $i \geq N$ and every halfspace $D$ delimited by one the $J_r$'s, $D$ belongs to the principal ultrafilter defined by $y_i$ if and only if $D \in y$ and similarly $D$ belongs to the principal ultrafilter defined by $z_i$ if and only if $D \in z$. It follows that the $J_r$'s separate $y_i$ and $z_i$ for every $i \geq N$. We also want to choose $N$ sufficiently large so that $(y_i,z_i)_{x_0} \geq D+2+3(L+3)$ for every $i \geq N$. Now, fix some $i \geq N$. As a consequence of Lemma \ref{lem:GromovProductEst}, we have
$$\delta_L(x_0,m(x_0,y_i,z_i)) \geq (y_i,z_i)_{x_0} - 3(L+3) \geq D+2.$$
Consequently, there exist $p \geq D+2$ pairwise $L$-well-separated hyperplanes $H_1, \ldots, H_p$ separating $x_0$ and $m:=m(x_0,y_i,z_i)$. Without loss of generality, suppose that $H_j$ separates $H_{j-1}$ and $H_{j+1}$ for every $2 \leq j \leq p-1$ and that $H_1$ separates $x_0$ from $H_p$. For every $1 \leq j \leq k$, notice that $J_j$ intersects the halfspace delimited by $H_p$ which contains $y_i$ and $z_i$ since it separates these two vertices; on the other hand, $J_j$ cannot be included into the halfspace delimited by $H_{D+1}$ which contains $m$ since the distance between $x_0$ and $J_j$ is at most $D$, so we conclude that $J_j$ must be transverse to $H_{D+1}$ and $H_{D+2}$. Since $H_{D+1}$ and $H_{D+2}$ are $L$-well-separated, and since $\{ J_1, \ldots, J_k \}$ does not contain a facing triple, we deduce that $k \leq L$. 

\medskip \noindent
The distance (possibly infinite) between $y$ and $z$ in the graph $\overline{X}$ being half the cardinality of the symmetric difference between $y$ and $z$, we conclude that $y$ and $z$ are at distance at most $L$ in $\overline{X}$. This concludes the proof of our claim, and finally of our proposition.
\end{proof}

\begin{cor}\label{cor:fixedpointXL}
Let $G$ be a group acting on some CAT(0) cube complex $X$. Fix some integer $L \geq 0$. If the induced action $G \curvearrowright (X,\delta_L)$ fixes a point at infinity, then $G$ stabilises a cube in the Roller boundary $\mathfrak{R}X$. 
\end{cor}

\begin{proof}
If $G$ fixes a point at infinity in $(X,\delta_L)$, it follows from Proposition \ref{prop:boundarymap} that $G$ stabilises some cubical component $Y$ of $\mathfrak{R}X$ and that the induced action $G \curvearrowright Y$ has a bounded orbit. Consequently, $G$ stabilises a cube in $\mathfrak{R}X$. 
\end{proof}

\begin{remark}
It can be shown that the set $R(\xi)$ we associated to a point $\xi \in \partial(X,\delta_L)$ in the proof of Proposition \ref{prop:boundarymap} is not only a small subset in a cubical component of $\mathfrak{R}X$ but it is a small cubical component: $R(\xi)$ is a cubical component of $\mathfrak{R}X$ of diameter at most $L$. As a consequence, the boundary of $(X,\delta_0)$ coincides with the set of strongly separated ultrafilters in $\mathfrak{R}X$ defined in \cite{PingPong} (and they have the same topology since they are both Cantor sets). However, we do not need this stronger statement, Proposition \ref{prop:boundarymap} will be sufficient for our purpose in Section \ref{section:cubicalrigidity}. 
\end{remark}

\section{Reducible elements in Thompson's group $V$}\label{section:reducible}

\noindent
This section is dedicated to the study of reducible elements (defined below) of Thompson's group $V$. It is the key starting point of our proof of the hyperbolic rigidity of $V$. First of all, let us recall the definition of $V$ as a homeomorphism group of the Cantor set. For more information, we refer to \cite{IntroThompson}.

\begin{definition}\label{def:V}
A \emph{dyadic decomposition} of $[0,1]$ is a collection of intervals $(I_k)$ of the form $\left[ \frac{j}{2^m}, \frac{j+1}{2^m} \right]$ covering $[0,1]$ such that the intersection between any two intervals contains at most one point. Given two dyadic decompositions $\mathfrak{A},\mathfrak{B}$ of $[0,1]$ and a bijection $\sigma : \mathfrak{A} \to \mathfrak{B}$, the map $\mathfrak{C} \to \mathfrak{C}$ defined on the Cantor set $\mathfrak{C} \subset [0,1]$ by sending $A \cap \mathfrak{C}$ to $\sigma(A) \cap \mathfrak{C}$ via an affine map induces a homeomorphism of $\mathfrak{C}$. \emph{Thompson's group} $V$ is the group of the homeomorphisms of $\mathfrak{C}$ which decompose in this way.
\end{definition}

\noindent
Here are the fundamental objects of our paper:

\begin{definition}
An element $g \in V$ is \emph{reducible} if there exists some non-trivial dyadic interval on which $g$ is the identity. Its \emph{thickness} is the maximal diameter of such an interval.
\end{definition}

\noindent
Our first lemma shows that $V$ is boundedly generated by reducible elements with controlled thickness.

\begin{lemma}\label{lem:reduciblegenerate}
Every element of $V$ is a product of four reducible elements of thickness at least $1/8$.
\end{lemma}
\begin{figure}
\begin{center}
\includegraphics[scale=0.16]{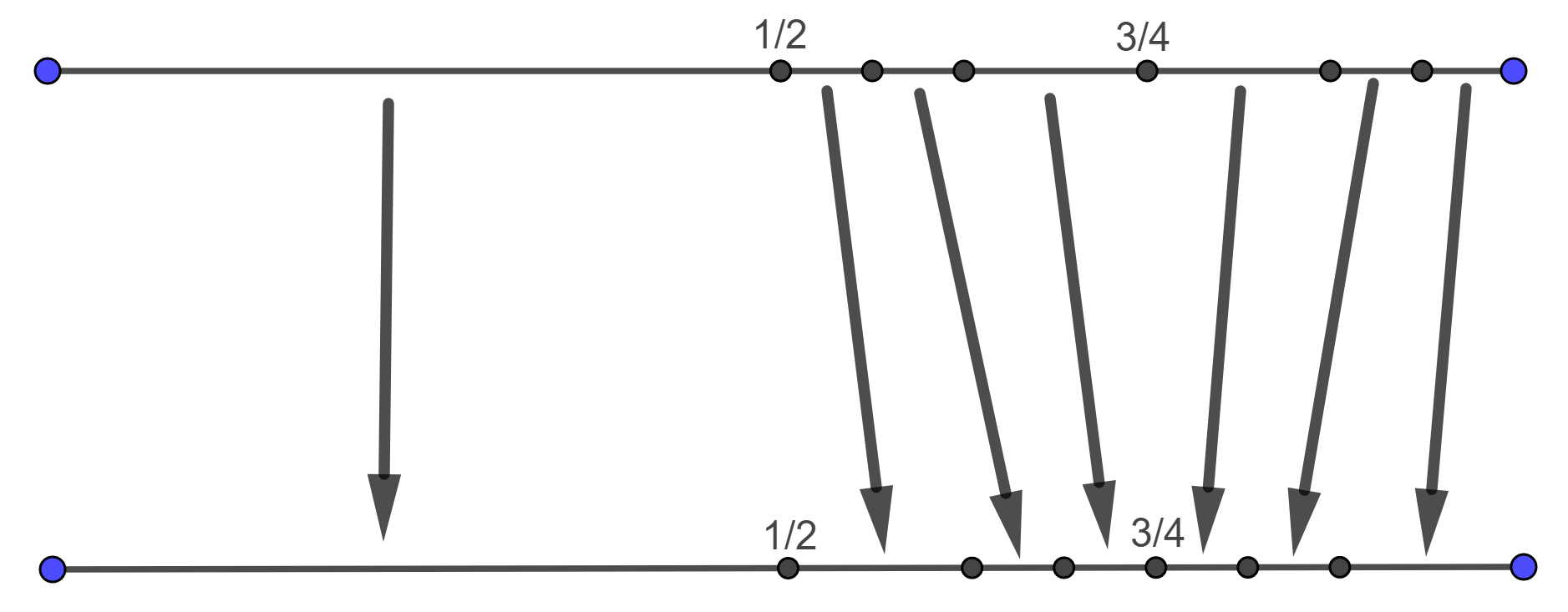}
\caption{The element $a$ in the proof of Lemma \ref{lem:reduciblegenerate}.}
\label{figure1}
\end{center}
\end{figure}
\begin{figure}
\begin{center}
\includegraphics[scale=0.16]{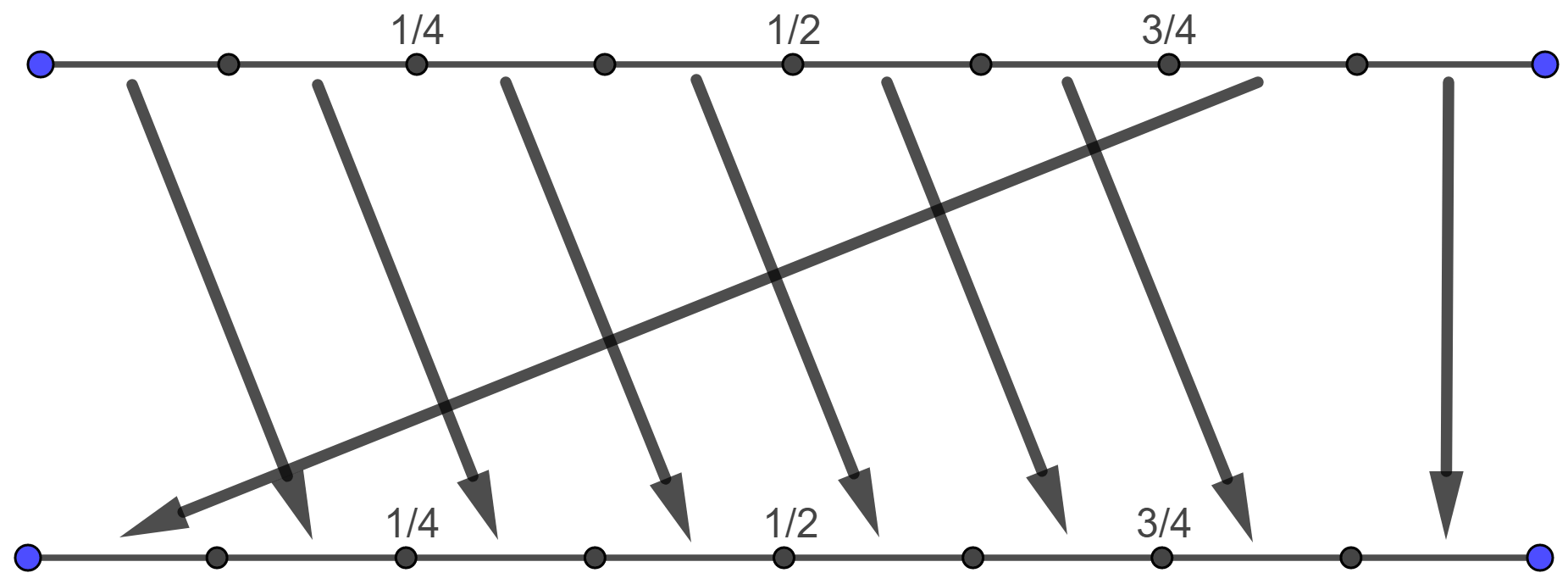}
\caption{The element $b$ in the proof of Lemma \ref{lem:reduciblegenerate}.}
\label{figure2}
\end{center}
\end{figure}
\begin{proof}
Let $a,b \in V$ be the elements defined by Figures \ref{figure1} and \ref{figure2} respectively. These elements satisfy the following properties:
\begin{itemize}
	\item $a$ fixes $3/4$ and, for every dyadic interval $I \subset (1/2,1)$, $\mathrm{length}(a^nI) \underset{n \to + \infty}{\longrightarrow} 0$;
	\item the restriction of $b$ over $[0,3/4]$ is a translation of length $1/8$.
\end{itemize}
Notice also that they are reducible elements of thickness at least $1/8$. Fix some $g \in V$.

\begin{claim}
There exist $n,m \in \mathbb{Z}$ such that $b^mga^n([3/4,7/8])$ is a dyadic interval contained into $[1/2,1]$. 
\end{claim}

\noindent
Without loss of generality, we may suppose that there exists a dyadic interval $A\subset (0,1)$ with $3/4$ as its left endpoint which is sent by $g$ to a dyadic interval $B$. Let $n \geq 0$ be sufficiently large so that $a^n([3/4,7/8]) \subset A$. As a consequence, $ga^n([3/4,7/8])$ is a dyadic interval. Moreover, we can also choose $n$ sufficiently large so that $ga^n([3/4,7/8])$ has length at most $1/8$. If $ga^n([3/4,7/8])$ is included into $[1/2,1]$, we are done. Otherwise, $ga^n([3/4,7/8])$ is included into $[0,5/8]$, we can translate it by a power of $b$, say $b^m$, into $[1/2,1]$. Thus, we have found $n,m \geq 0$ such that $b^mga^n([3/4,7/8]) \subset [1/2,1]$, proving our claim. 

\medskip \noindent
Now, let $c \in V$ be any element which sends the dyadic interval $b^mga^n([3/4,7/8])$ to $[3/4,7/8]$ and which is the identity over $[0,1/4]$. Then $c$ is a reduced element of thickness at least $1/4$, and by construction $cb^mga^n$ fixes $[3/4,7/8]$ so that it must be a reduced element of thickness at least $1/8$. Our lemma follows from the equality
$$g = b^{-m} \cdot c^{-1} \cdot cb^mga^n \cdot a^{-n}$$
and from the observation that any power of a reducible element of thickness at least $1/8$ is again a reduced element of thickness at least $1/8$.
\end{proof}

\noindent
Our second lemma essentially shows that any reducible element generates a direct product with at least one of its conjugates.

\begin{lemma}\label{lem:conjdisjointsupp}
Let $g \in V$ be a reducible element and $I \subset \mathrm{Fix}(g)$ a dyadic interval. There exists some $h \in V$ such that $\mathrm{supp}(hgh^{-1}) \subset I$.
\end{lemma}
\begin{figure}
\begin{center}
\includegraphics[scale=0.16]{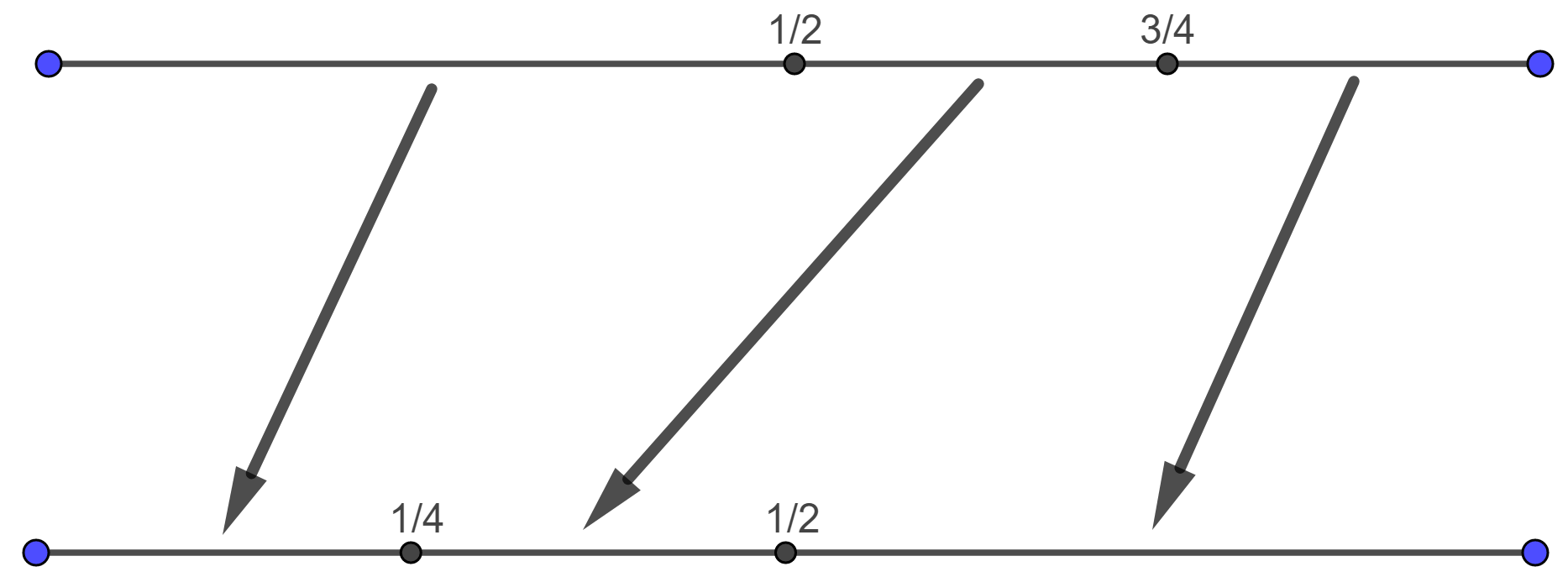}
\caption{The element $d$ in the proof of Lemma \ref{lem:conjdisjointsupp}.}
\label{figure3}
\end{center}
\end{figure}
\begin{proof}
Let $a \in V$ be a permutation sending the dyadic interval $I$ to a dyadic interval $J$ containing $1$, and let $d \in V$ be the element defined by Figure \ref{figure3}. Notice that $d$ satisfies the following property: for every dyadic interval $K$ containing $1$, $\mathrm{length}(d^nK) \underset{n \to + \infty}{\longrightarrow} 1$. Therefore, there exists some $n \geq 0$ such that $d^nJ$ has length at least $1- \mathrm{length}(I)/2$. As a consequence, the image of $\mathrm{supp}(g)$ by $d^na$ has diameter at most $\mathrm{length}(I)/2$. It follows that there exists a permutation $b \in V$ such that $bd^na$ sends $\mathrm{supp}(g)$ inside $I$. Thus, if we set $h=bd^na$ then
$$\mathrm{supp}(hgh^{-1})= h \cdot \mathrm{supp}(g) \subset I,$$
which concludes the proof of our lemma.
\end{proof}

\noindent
Our third lemma shows that two arbitrary elements of $V$ can be made reducible simultaneously in many different ways.

\begin{lemma}\label{lem:simult}
For every $g_1,g_2 \in V$, there exist a reducible $h \in V$ and a dyadic interval $I \subset (0,1)$, such that $fh$, $fhg_1$ and $fhg_2$ are all reducible for every $f \in V$ satisfying $\mathrm{supp}(f) \subset I$. 
\end{lemma}

\begin{proof}
Let $g_1,g_2 \in V$ be two elements.

\begin{claim}
There exist two disjoint dyadic intervals $A,B$ such that $g_1(A)$ and $g_2(B)$ are also two disjoint dyadic intervals.
\end{claim}

\noindent
Fix two disjoint dyadic interval $A_0,B_0$ such that $A_0$ contains $0$ and $B_0$ contains $1$, and such that $g_1(A_0)$ and $g_2(B_0)$ are dyadic intervals. If one the endpoints of $g_1(A_0)$ does not belong to $g_2(B_0)$, then there exists a dyadic subinterval $A \subset A_0$ such that $g_1(A)$ and $g_2(B_0)$ are disjoint, and we are done. Similarly, if one of the endpoints of $g_2(B_0)$ does not belong to $g_1(A_0)$, then there exists a dyadic subinterval $B \subset B_0$ such that $g_1(A_0)$ and $g_2(B)$ are disjoint, and we are done. The only remaining case to consider is $g_1(A_0)=g_2(B_0)$. Here, set $A$ as the first fourth of $A_0$ and $B$ as the last fourth of $B_0$. Then $g_1(A)$ and $g_2(B)$ are disjoint. This concludes the proof of our claim.

\medskip \noindent
Now, fix two dyadic intervals $A,B$ given by our claim. Notice that, up to replacing $A$ with one of its halves, we may suppose without loss of generality that $A \cup B \cup g_1(A) \cup g_2(B)$ does not cover all the dyadic intervals of $[0,1]$. Fix a dyadic interval disjoint from $A \cup B \cup g_1(A) \cup g_2(B)$, and let $I,J$ denote its two halves. Because $A$, $B$, $g_1(A)$, $g_2(B)$, $I$ and $J$ are pairwise disjoint, it follows that there exists some $h \in V$ sending $g_1(A)$ to $A$, $g_2(B)$ to $B$ and fixing $I \cup J$. Notice that $h$ is a reducible element. Now, let $f \in V$ be an arbitrary element satisfying $\mathrm{supp}(f) \subset I$. Notice that
$$fhg_1(A)=f(A)=A, \ fhg_2(B)=f(B)=B \ \text{and} \ fh(J)=f(J)=J,$$
so $fh$, $fhg_1$ and $fhg_2$ are all reducible. This concludes the proof of our lemma. 
\end{proof}

\noindent
Finally, our fourth and last lemma shows how to conjugate a reducible element, in a controlled way, in order to include its support into a given dyadic interval.

\begin{lemma}\label{lem:bigF}
For every dyadic interval $I \subset (0,1)$ and every $\epsilon>0$, there exist reducible $h_1,\ldots, h_8 \in V$ such that $F := \langle h_1 \rangle \cdots \langle h_8 \rangle$ satisfies the following. For every reducible element $g \in V$ of thickness at least $\epsilon$, there exists some $f \in F$ such that $fgf^{-1}$ fixes $I^c$. 
\end{lemma}
\begin{figure}
\begin{center}
\includegraphics[scale=0.2]{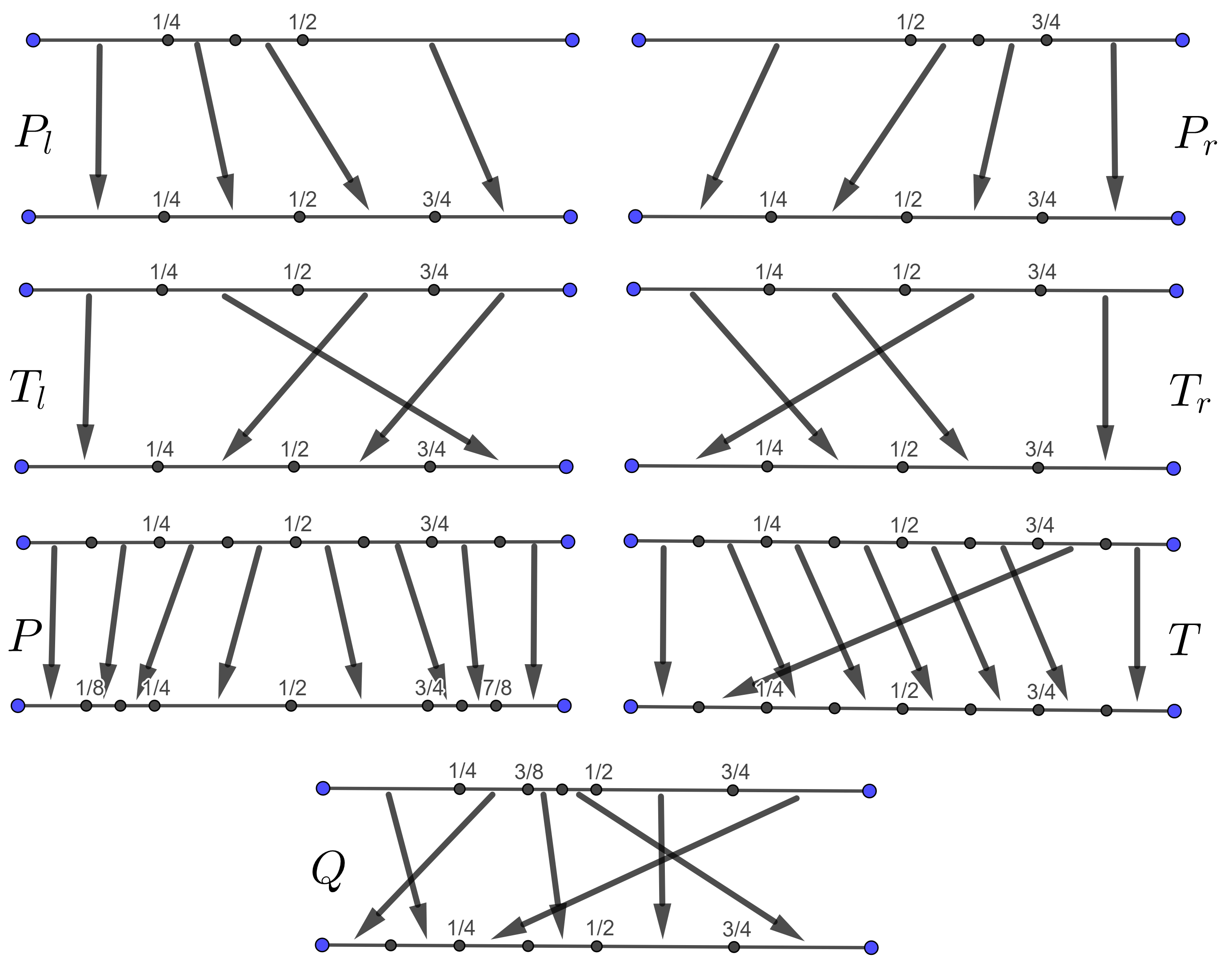}
\caption{Examples of elements $P_l, P_r, T_l,T_r, P,T,Q \in V$ from the proof of Lemma \ref{lem:bigF}.}
\label{GrosseFigure}
\end{center}
\end{figure}
\begin{proof}
Fix a dyadic interval $I \subset (0,1)$, and real number $\epsilon>0$. Without loss of generality, we will suppose that $\epsilon$ is a negative power of two which is sufficiently small so that $[0,\epsilon]$ and $[1-\epsilon,\epsilon]$ are included into $I^c$ and so that $\epsilon < \mathrm{length}(I)/2$. Fix some elements $P_l, P_r, T_l,T_r, P,T,Q \in V$ satisfying the following properties (such elements are illustrated by Figure \ref{GrosseFigure} for $\epsilon=1/2$):
\begin{itemize}
	\item $P_l$ fixes $[0,\epsilon/2]$, and $\mathrm{length}(P_l^nK) \underset{n \to + \infty}{\longrightarrow} 1$ for every dyadic interval $K$ containing $0$ and $\epsilon$;
	\item $P_r$ fixes $[1,1-\epsilon/2]$, and $\mathrm{length}(P_l^nK) \underset{n \to + \infty}{\longrightarrow} 1$ for every dyadic interval $K$ containing $1$ and $1-\epsilon$;
	\item $T_l$ fixes $[0,\epsilon/2]$ and is a translation of length $\epsilon/2$ to the left on $[\epsilon,1]$;
	\item $T_r$ fixes $[1,1-\epsilon/2]$ and is a translation of length $\epsilon/2$ to the right on $[0,1-\epsilon]$;
	\item $P$ fixes $[0,\epsilon/4] \cup [1-\epsilon/4,1]$, and $\mathrm{length}(P^nK) \underset{n \to + \infty}{\longrightarrow} 1-\epsilon/2$ for every dyadic interval $K$ containing $1/2$ in its interior;
	\item $T$ fixes $[0,\epsilon/4] \cup [1-\epsilon/4,1]$ and is a translation of length $\epsilon/4$ to the right on $[\epsilon/4,1-\epsilon/2]$;
	\item $Q$ sends $[0,\epsilon/2] \cup [1-\epsilon/2,1]$ into $[\epsilon/4,1/2-\epsilon/4]$ and is the identity on $[1/2,1-\epsilon/2]$.
%	\item $Q$ sends $[0,\epsilon/3] \cup [1-\epsilon/3,1]$ into $[\epsilon/4,\epsilon/4+2\epsilon /3]$ and is the identity on $[1/2,1-\epsilon/3]$.
\end{itemize}
Let $g \in V$ be a reduced element of thickness at least $\epsilon$, and let $J \subset [0,1]$ be a dyadic interval of length at least $\epsilon$ on which $g$ is the identity. Setting 
$$F = \langle T_l \rangle \cdot \langle P_l \rangle \cdot \langle T_r \rangle \cdot \langle P_r \rangle \cdot \langle T \rangle \cdot \langle Q \rangle \cdot \langle P \rangle \cdot \langle T \rangle,$$ 
our goal is to show that $fgf^{-1} \in \mathrm{Fix}(I^c)$ for some $f \in F$. If $g=1$, there is nothing to prove, so from now on we suppose that $g\neq 1$.

\medskip \noindent
\textbf{Case 1:} $0$ belongs to $J$. Notice that $\epsilon$ belongs to $J$ since $J$ has length at least $\epsilon$. Consequently, there exists some $n \geq 0$ such that $P_l^nJ$ has length at least $1-\mathrm{length}(I)/2$, so that $\mathrm{supp}(P_l^ngP_l^{-n})=P_l^n \mathrm{supp}(g)$ is included into the dyadic segment $K$ which contains $1$ and which has length at most $\mathrm{length}(I)/2$. Now, let $m \geq 0$ be such that $T_l^m(K) \subset I$. So
$$\mathrm{supp}(T_l^mP_l^n g P_l^{-n}T_l^{-m}) = T_l^mP_l^n \mathrm{supp}(g) \subset T_l^m(K) \subset I,$$
which shows that $T_l^mP_l^n g P_l^{-n}T_l^{-m}$ belongs to $\mathrm{Fix}(I^c)$. 

\medskip \noindent
\textbf{Case 2:} $1$ belongs to $J$. The situation is symmetric to the previous one: just replace $P_l$ and $T_l$ with $P_r$ and $T_r$ respectively. 

\medskip \noindent
\textbf{Case 3:} $J$ does not contain $0$ and $1$. Up to extracting a dyadic subinterval from $J$, we may suppose that $J$ is disjoint from $[0,\epsilon/4] \cup [1-\epsilon/4,1]$ and has length at least $\epsilon/2$. There exists some $n \in \mathbb{Z}$ such that $T^n(J)$ contains $1/2$ in its interior. Next, there exists some $m \geq 0$ such that $P^mT^n(J)$ has length at least $1-2\epsilon/3$. Consequently, the support of $P^mT^ngT^{-n}P^{-m}$ is included into $[0,\epsilon/3] \cup [1-\epsilon/3,1]$, and a fortiori into $[0,\epsilon/2] \cup [1-\epsilon/2,1]$. Because $Q$ sends $[0,\epsilon/2] \cup [1-\epsilon/2,1]$ into an interval $K$ of length at most $2\epsilon/3$ inside $[\epsilon/4,1/2-\epsilon/4]$, there exists some $p \in \mathbb{Z}$ such that $T^p(K) \subset I$. One has
$$\begin{array}{lcl} \mathrm{supp}(T^pQ P^mT^n g T^{-n}P^{-m}Q^{-1}T^{-p}) & = & T^pQ \cdot \mathrm{supp}( P^mT^n g T^{-n}P^{-m}) \\ \\ & \subset & \subset T^pQ \left( [0,\epsilon/3] \cup [1-\epsilon/3,1] \right) \\ \\ & \subset & T^p \left(K \right) \subset I \end{array}$$
hence $T^ pQP^mT^n g T^{-n}P^{-m}Q^{-1}T^{-p} \in \mathrm{Fix}(I^c)$.
\end{proof}

\section{Hyperbolic rigidity}\label{section:hyprigidity}

\noindent
We begin this section by proving the main criterion of our article, namely:

\begin{thm}\label{thm:maincriterion}
Let $G$ be a group. Suppose that there exist two subsets $A \subset B \subset G$ satisfying the following conditions.
\begin{itemize}
	\item $G$ is \emph{boundedly generated} by $A$, ie., there exists some $N \geq 0$ such that every element of $G$ is the product of at most $N$ elements of $A$.
	\item For every $a,b \in B$, there exist $g,h \in G$ such that $$[gag^{-1},a]=[gag^{-1},hgag^{-1}h^{-1}]=[hgag^{-1}h^{-1},b]=1.$$
	\item For every $a,b \in G$, there exist some $h, h_1, \ldots, h_r \in B$ such that the following holds. For every $k \in A$, there exists some $f \in \langle h_1 \rangle \cdots \langle h_r \rangle$ such that the elements $fkf^{-1}h$, $fkf^{-1}ha$ and $fkf^{-1}hb$ all belong to $B$.
\end{itemize}
Then any isometric action of $G$ on a hyperbolic space fixes a point at infinity, or stabilises a pair of points at infinity, or has bounded orbits. 
\end{thm}

%\noindent
%[Correction added on 9 december 2018 after online publication: The second item has been modified (producing a more general statement). The previous statement and its proof were correct, but the new formulation allowed us to correct a mistake in the proof of Theorem \ref{thm:VhypRigid} below.]

%\medskip 
\noindent
From now on, we fix a group $G$ and two subsets $A \subset B \subset G$ satisfying the above conditions. We recall from Section \ref{section:hyp} that we may suppose without loss of generality that our hyperbolic spaces are graphs. Our statement will be an easy consequence of the following two lemmas.

\begin{lemma}\label{lem:reducibleelliptic}
Let $G$ act on some hyperbolic graph. If $G$ does not fix a point at infinity nor stabilises a pair of points at infinity, then all the elements of $B$ are elliptic.
\end{lemma}

\begin{proof}
Suppose that there is some element $a \in B$ which is not elliptic. Let $\partial$ be the set of points at infinity fixed by $a$ (so $\partial$ has cardinality one if $a$ is parabolic, or two if $a$ is loxodromic). Given any other element $b \in B$, we claim that $b$ stabilises $\partial$. 

\medskip \noindent
By assumption, we know that there exist $g,h \in G$ such that $$[gag^{-1},a]=[gag^{-1},hgag^{-1}h^{-1}]=[hgag^{-1}h^{-1},b]=1.$$ For convenience, set $\bar{a}=gag^{-1}$. Notice that, as $a$ and $\bar{a}$ are conjugate and commute, the sets of points at infinity fixed by $a$ and $\bar{a}$ coincide, ie., $\partial$ is also the set of points at infinity fixed by $\bar{a}$. Similarly, as $\bar{a}$ and $h\bar{a}h^{-1}$ are conjugate and commute, $\partial$ coincides with the set of points at infinity fixed by $h \bar{a} h^{-1}$. Next, because $b$ and $h\bar{a}h^{-1}$ commute, it follows that $b$ has to stabilise $\partial$, concluding the proof of our claim. 

\medskip \noindent
Since $G$ is generated by $B$ (as $B$ contains the generating set $A$), it follows that $G$ stabilises $\partial$. Consequently, $G$ fixes a point at infinity or stabilises a pair of points at infinity.
\end{proof}

%\noindent
%[Correction added on 9 december 2018 after online publication: The proof of Lemma~\ref{lem:reducibleelliptic} has been adapted to the new formulation of Theorem~\ref{thm:maincriterion}, but the argument remains essentially the same.]

\begin{lemma}\label{lem:Vboundedorbit}
Let $G$ act on some $\delta$-hyperbolic graph $X$. If the action does not fix a point at infinity nor stabilises a pair of points at infinity, and if all the elements of $B$ are elliptic, then $G$ has bounded orbits.
\end{lemma}

\begin{proof}
Suppose by contradiction that $G$ has unbounded orbits and does not fix any point at infinity. As a consequence, there exist two independent loxodromic isometries $g_1,g_2 \in G$ (see \cite[Paragraph 8.2.E]{GromovHyp}). Fix two axes $\ell_1,\ell_2$ of $g_1,g_2$ respectively. Let $h, h_1,\ldots,h_r \in B$ be the elements given in the statement of Theorem \ref{thm:maincriterion}. By assumptions, the $h_i$'s are elliptic, so, as a consequence of \cite[Lemma III.$\Gamma$.3.3]{BH}, for every $1 \leq i \leq r$ there exists some $x_i \in X$ such that the orbit $\langle h_i \rangle \cdot x_i$ has diameter at most $17\delta$. Set $F= \langle h_1 \rangle \cdots \langle h_r \rangle$. 

\begin{claim}\label{claim:reducibleboundedorbit}
Fix a basepoint $x_0 \in X$. For every element $k \in A$,
$$d(x_0,k \cdot x_0) \leq  (19+8r)D+\Delta+ 4(2+17r)\delta$$
where 
$$D= \max \left\{ d(x_0,x_1), \ldots, d(x_0,x_r), d(x_0, \mathrm{proj}_{\ell_2}(\ell_1))+\mathrm{diam}(  \mathrm{proj}_{\ell_2}(\ell_1)),d(x_0,hx_0) \right\}$$ 
and $M= M(\delta,1,62\delta)$ is the Morse constant.
\end{claim}

\noindent
Fix some $k \in A$. By assumption, there exists some $f \in F$ such that $fkf^{-1}h$, $fkf^{-1}hg_1$ and $fkf^{-1}hg_2$ all belong to $B$. As a consequence, they are elliptic isometries. It follows from Lemma \ref{lem:hypgeom} that there exist points $x_1 \in \ell_1$ and $x_2 \in \ell_2$ which are moved within distance at most $\Delta:= 8 M+243\delta$ by $fkf^{-1}h$. Fix a geodesic $[x_1,x_2]$ in $X$ between $x_1$ and $x_2$. As a consequence of Lemma \ref{lem:geodinterproj}, $[x_1,x_2]$ intersects the $M$-neighborhood of the nearest-point projection of $\ell_1$ onto $\ell_2$. Let $x$ be a point which belongs to this intersection. By $8\delta$-convexity of the metric (see \cite[Corollary 10.5.3]{CDP}), $fkf^{-1}h$ moves $x$ within distance at most $\Delta +8\delta$. On the other hand,
$$\begin{array}{lcl} d(fkf^{-1}h \cdot x,x) & = & d( f^{-1}hf k \cdot f^{-1}hx,f^{-1}hx) \\ \\ & \geq & d(f^{-1}hf k \cdot x,x)-2d(f^{-1}h \cdot x,x) \\ \\ & \geq  &d(k \cdot x,f^{-1} h^{-1}f \cdot x) - 2d(f^{-1}h \cdot x,x) \\ \\ & \geq & d(k \cdot x,x) - d(f^{-1}h^{-1}f \cdot x,x) - 2d(f^{-1}h \cdot x,x) \end{array}$$
Consequently,
$$\begin{array}{lcl} d(x,k \cdot x) & \leq & d(f^{-1}h^{-1}f \cdot x,x) + 2d(f^{-1}h \cdot x,x)+\Delta +8\delta \\ \\ & \leq & 4d(x,fx) +3 d(x,hx)+\Delta+8\delta \\ \\ & \leq & 4d(x_0,fx_0)+3 d(x_0,hx_0) + 14 d(x_0,x)+ \Delta +8\delta \end{array}$$
By combining the observation that 
$$d(x_0,x) \leq d(x_0, \mathrm{proj}_{\ell_2}(\ell_1)) +\mathrm{diam}(  \mathrm{proj}_{\ell_2}(\ell_1)) \leq D$$
together with the next claim, we deduce that
$$\begin{array}{lcl} d(x_0,k \cdot x_0) & \leq & d(x,k \cdot x)+2 d(x,x_0) \leq 4 \cdot r(2D+17\delta)+19D +\Delta+8\delta \\ \\ & \leq & (19+8r)D+\Delta+ 4(2+17r)\delta \end{array}$$
concluding the proof of our claim.

\begin{claim}
For every $f \in F$, $d(x_0, fx_0) \leq r(2D+17\delta)$. 
\end{claim}

\noindent
Write $f=h_1^{n_1} \cdots h_r^{n_r}$ for some $n_1, \ldots, n_r \in \mathbb{Z}$. Then
$$\begin{array}{lcl} d(x_0,fx_0) & \leq & \displaystyle \sum\limits_{i=1}^r d(h_i^{n_i}x_0,x_0) \leq \sum\limits_{i=1}^r \left( d(h_i^{n_i}x_i,x_i) + 2d(x_i,x_0) \right) \\ \\ & \leq & \displaystyle 2r \max\limits_{1 \leq i \leq r} d(x_0,x_i) + \sum\limits_{i=1}^r d(h_i^{n_i}x_i,x_i) \\ \\ & \leq & 2r \max\limits_{1 \leq i \leq r} d(x_0,x_i) + 17r \delta  \end{array}$$
This proves the claim. 

\medskip \noindent
Now, we are ready to conclude the proof of our lemma. Indeed, by combining Lemma \ref{lem:reduciblegenerate} with Claim \ref{claim:reducibleboundedorbit}, it follows that there exists some constant $K$ such that $d(gx_0,x_0) \leq K$ for every $g \in G$. Otherwise saying, $G$ has a bounded orbit, contradicting our starting hypothesis. 
\end{proof}

\begin{proof}[Proof of Theorem \ref{thm:maincriterion}.]
Let $X$ be a hyperbolic graph on which $G$ acts. If the action fixes a point at infinity or stabilises a pair of points at infinity, we are done. Otherwise, it follows from Lemma \ref{lem:reducibleelliptic} that the elements of $B$ must be elliptic, and we conclude from Lemma \ref{lem:Vboundedorbit} that $G$ has a bounded orbit. Thus, we have proved the desired statement for hyperbolic graphs. But the general case reduces to hyperbolic graphs according to Lemma \ref{lem:hypgraph}, so the proof is concluded. 
\end{proof}

\noindent
Now, we are ready to prove that Thompson's group $V$ is hyperbolically elementary.

\begin{thm}\label{thm:VhypRigid}
Any isometric action of Thompson's group $V$ on a Gromov-hyperbolic space either fixes a unique point at infinity or has a bounded orbit.
\end{thm}

\begin{proof}
We claim that $V$ satisfies the hypotheses of Theorem \ref{thm:maincriterion} if $B$ denotes the set of reducible elements and $A$ the set of reducible elements of thickness at least $1/8$. The first item of Theorem \ref{thm:maincriterion} is a direct consequence of Lemma \ref{lem:reduciblegenerate}.

\medskip \noindent
Next, let $a,b \in V$ be two reducible elements. If the supports of $a$ and $b$ do not cover the Cantor set $\mathfrak{C}$, set $g=1$. Otherwise, if the supports of $a$ and $b$ cover $\mathfrak{C}$, there exist two disjoint dyadic intervals $I$ and $J$ on which $a$ and $b$ respectively are the identity. According to Lemma \ref{lem:conjdisjointsupp}, there exists some element $g \in V$ such that the support of $gag^{-1}$ is included into $I$. Now, the point is that $a$ and $gag^{-1}$ commute (since they have disjoint supports) and that $gag^{-1}$ and $b$ are both the identity on some dyadic interval $J$. Again according to Lemma \ref{lem:conjdisjointsupp}, there exists some $h \in V$ such that the support of $hgag^{-1}h^{-1}$ is included into $J$. By construction, the support of $hgag^{-1}h^{-1}$ is disjoint from the supports of $b$ and $gag^{-1}$, so that $hgag^{-1}h^{-1}$ has to commute with both $b$ and $gag^{-1}$. This proves the second item of Theorem \ref{thm:maincriterion}.

%\medskip \noindent
%[Correction added on 9 december 2018 after online publication: The previous paragraph has been rewritten to correct a mistake.]

\medskip \noindent
Finally, Fix two elements $a,b \in V$. Let $h \in V$ and $I \subset (0,1)$ be the element of $V$ and the dyadic interval given by Lemma \ref{lem:simult}, and let $h_1, \ldots, h_8 \in V$ be the elements given by Lemma \ref{lem:bigF} for $I$ and $\epsilon=1/8$. Set $F= \langle h_1 \rangle \cdots \langle h_8 \rangle$. Given a reducible element $k \in V$ of thickness at least $1/8$, we deduce from Lemma \ref{lem:bigF} there exists some $f \in F$ such that the support of $fkf^{-1}$ is included into $I$. It follows from Lemma \ref{lem:simult} that $fkf^{-1}h$, $fkf^{-1}ha$ and $fkf^{-1}hb$ are all reducible elements. This proves the third item of Theorem \ref{thm:maincriterion}.

\medskip \noindent
Therefore, Theorem \ref{thm:maincriterion} applies, proving that that any isometric action of $V$ on a hyperbolic space fixes a unique point at infinity, or stabilises a pair of points at infinity, or has a bounded orbit. It remains to show that $V$ cannot stabilise a pair of points at infinity, or equivalently: 

\begin{claim}
Any isometric action of $V$ on a quasi-line must have a bounded orbit. 
\end{claim}

\noindent
Since an action on a quasi-line with an unbounded orbit produces a \emph{quasi-morphism} (ie., a map $\varphi : V \to \mathbb{R}$ satisfying the following conditions: there exists a constant $D \geq 0$ such that $|\varphi(gh) - \varphi(g)- \varphi(h)| \leq D$ for every $g,h \in V$) which is \emph{unbounded} (ie., such that $\varphi(V)$ is not bounded in $\mathbb{R}$), it is sufficient to show that any quasi-morphism of $V$ is necessarily bounded. We refer to \cite{CalegariScl} for more information on quasi-morphisms. The last observation is a straightforward consquence of the fact that $V$ is \emph{uniformly simple} \cite[Corollary 6.6]{VuniformSimple}, meaning that there exists a constant $N \geq 1$ such that, for every non-trivial elements $f,g \in V$, $f$ can be written as a product of at most $N$ conjugates of $g$ or $g^{-1}$. 
\end{proof}

\section{Cubical rigidity}\label{section:cubicalrigidity}

\noindent
Our last section is dedicated to cubical rigidity, and more precisely, how to deduce it from hyperbolic rigidity. 

\begin{thm}\label{thm:FWinfty}
A finitely generated group all of whose finite-index subgroups 
\begin{itemize}
	\item are hyperbolically elementary,
	\item and do not surject onto $\mathbb{Z}$,
\end{itemize}
satisfies Property $(FW_{\infty})$.
\end{thm}

\begin{proof}
We want to prove by induction that, for every $n \geq 0$, a group of all whose finite-index subgroups are hyperbolically elementary and do not surject onto $\mathbb{Z}$ satisfies Property $(FW_n)$. For $n=0$, there is nothing to prove; so suppose that our statement is true for some $n \geq 0$, and fix a group $G$, all of whose finite-index subgroups are hyperbolically elementary and do not surject onto $\mathbb{Z}$, acting on a $(n+1)$-dimensional CAT(0) cube complex $X$. 

\medskip \noindent
Suppose first that $G$ fixes a point at infinity in $X$ (ie., in the visual boundary). It follows from \cite[Proposition 2.26]{CIF} that $G$ contains a finite-index subgroup $H$ which stabilises a cubical component $Y \subset \mathfrak{R}X$. As $\dim(Y)<\dim(X)$ according to Lemma \ref{lem:dimcubicalcomp}, our induction hypothesis implies that $H$ fixes a point of $Y$, so that $H$ fixes a point in $\mathfrak{R}X$; up to taking a finite-index subgroup of $H$, we may suppose without loss of generality that $H$ fixes a vertex in the Roller boundary. It follows from \cite[Theorem B.1]{CIF} that $H$ virtually surjects onto a free abelian group of rank $k \leq \dim(X)$ with a kernel which is locally elliptic (in $X$). Since the finite-index subgroups of $G$ do not surject onto $\mathbb{Z}$, necessarily $k=0$, so $H$ is virtually locally elliptic, and finally elliptic since $H$ is finitely generated. We conclude that $G$ has a bounded orbit in $X$, and finally that it fixes a point. Notice that we have proved the following statement, which we record for future use:

\begin{fact}\label{fact:interiorfromboundary}
If $G$ contains a finite-index subgroup fixing a vertex of $\mathfrak{R}X$ then $G$ has to fix a point of $X$.
\end{fact}

\noindent
From now on, suppose that $G$ does not fix a point at infinity in $X$. According to \cite[Proposition 3.5]{CapraceSageev}, up to taking a convex subcomplex of $X$, we may suppose that the action is essential. If $X$ is bounded, then $G$ fixes a point, so suppose that $X$ is unbounded. As a consequence of \cite[Proposition 2.6]{CapraceSageev}, $X$ decomposes as a product of irreducible CAT(0) cube complexes $X_1 \times \cdots \times X_r$ and $G$ contains a finite-index subgroup $H$ lying in $\mathrm{Isom}(X_1) \times \cdots \times \mathrm{Isom}(X_r)$. If $r \geq 2$ (ie., if $X$ is reducible), then $\dim(X_i) < \dim(X)$ for every $1 \leq i \leq r$, so that our induction hypothesis implies that all the induced actions $H \curvearrowright X_i$ have global fixed points. Consequently, $H$ fixes a point in $X$, and it follows that $G$ has a bounded orbit in $X$, and finally that it fixes a point. 

\medskip \noindent
From now on, suppose that $X$ is irreducible. It follows from \cite[Theorem 6.3]{CapraceSageev} that $G$ contains a contracting isometry of $X$, so that Theorem \ref{thm:hypmodel} implies that there exists some $L \geq 0$ such that $G$ acts on the hyperbolic space $(X,\delta_L)$ defined in Section \ref{section:cubecomplexes} with a loxodromic isometry. Because $G$ is hyperbolically elementary, it must contain a subgroup $G'$ of index at most two which fixes a point at infinity in $(X,\delta_L)$. It implies, according to Corollary \ref{cor:fixedpointXL}, that $G'$ stabilises a cube in the Roller boundary of $X$, so that some finite-index subgroup $H \subset G$ fixes a vertex in $\mathfrak{R}X$. We conclude from Fact \ref{fact:interiorfromboundary} that $G$ fixes a point in $X$. 

\medskip \noindent
Thus, we have proved that $G$ necessarily fixes a point of $X$. This concludes the proof of our theorem. 
\end{proof}

\begin{proof}[Proof of Corollary \ref{cor:FWforV}.]
First of all, since $V$ is a finitely generated simple group, it does not contain any proper finite-index subgroup. We also know from Theorem \ref{thm:VhypRigid} that $V$ is hyperbolically elementary, and, once again because $V$ is a simple group, it does not surject onto $\mathbb{Z}$. Consequently, Theorem \ref{thm:FWinfty} applies, implying that $V$ satisfies Property~$(FW_{\infty})$. 
\end{proof}

\noindent
We conclude this section by an example, which was communicated to us by Pierre-Emmanuel Caprace, showing that the property of being hyperbolically elementary is not stable under taking finite-index subgroups. Consequently, in order to apply Theorem \ref{thm:FWinfty}, we really need to check the hyperbolic elementarity for all the finite-index subgroups of the group we are looking at. 

\begin{ex}\label{ex:Caprace}
Let $H$ be a non-elementary hyperbolic group, say a free group of rank two. Then the wreath product $H \wr \mathbb{Z}_2 = (H \times H) \rtimes \mathbb{Z}_2$ contains the group $H \times H$ as a finite-index subgroup which is not hyperbolically elementary, but turns out to be hyperbolically elementary itself as justified by the argument below.

\medskip \noindent
Let $H \wr \mathbb{Z}_2$ act on a Gromov-hyperbolic space $X$. For convenience, let $H_1$ (resp. $H_2$) denote the first copy of $H$ (resp. the second copy of $H$) in the decomposition $H \wr \mathbb{Z}_2 = (H \times H) \rtimes \mathbb{Z}_2$. Notice that, if $H_1$ has a bounded orbit, then so does $H_2$ since $H_1$ and $H_2$ are conjugate. Consequently, if $H_1$ has a bounded orbit then so does $H \wr \mathbb{Z}_2$. From now on, we suppose that $H_1$ has unbounded orbits. 

\medskip \noindent
Suppose that $H_1$ contains two independent loxodromic isometries $f$ and $h$. Let $\partial$ denote the union of the two pairs of points at infinity stabilised by $f$ and $h$. If $f'$ denote the element of $H_2$ conjugate to $f$, then $f'$ has to stabilise $\partial$ as it commutes with both $f$ and $g$. It implies that $f'$ must be elliptic, which is impossible since it is conjugate to $f$ which loxodromic. Thus, $H_1$ cannot contain two independent loxodromic isometries. 

\medskip \noindent
It follows that $H_1$ fixes a unique point at infinity $\xi$ or stabilises a pair of points at infinity $\{\xi_1,\xi_2\}$; set $\partial = \{\xi\}$ or $\partial = \{\xi_1,\xi_2\}$ depending on the situation. Notice that, for every $h \in H_2$, $h \partial$ is also stabilised by $H_1$ since $h$ commutes with any element of $H_1$. On the other hand, $H_1$ does not have a bounded orbit so it cannot stabilise a subset of cardinality at least three in the boundary of $X$, hence $h \partial = \partial$ for every $h \in H_2$. Therefore, $\partial$ is stabilised by both $H_1$ and $H_2$. It follows that $H \wr \mathbb{Z}_2$ stabilises a finite set in $\partial X$. If this subset has cardinality at least three, then $H \wr \mathbb{Z}_2$ has a bounded orbit; otherwise, $H \wr\mathbb{Z}_2$ fixes a point at infinity or stabilises a pair of points at infinity.

\medskip \noindent
Thus, we have proved that $H \wr \mathbb{Z}_2$ is hyperbolically elementary. 
\end{ex}

\addcontentsline{toc}{section}{References}

\bibliographystyle{alpha}
\bibliography{HypAndCubRigidityV}

\end{document}